\documentclass[onefignum,onetabnum]{siamart220329}

\usepackage{amsfonts}
\usepackage{graphicx}
\usepackage{epstopdf}
\usepackage{algorithmic}
\usepackage{mathabx}
\usepackage{braket,amsfonts,amsopn} 
 
\usepackage{enumitem}
\usepackage{enumerate}
 
\newtheorem{example}{Example}[section]
\newtheorem{assumption}{Assumption}[section]

\ifpdf
  \DeclareGraphicsExtensions{.eps,.pdf,.png,.jpg}
\else
  \DeclareGraphicsExtensions{.eps}
\fi

\newsiamremark{remark}{Remark}
\newsiamremark{hypothesis}{Hypothesis}
\crefname{hypothesis}{Hypothesis}{Hypotheses}
\newsiamthm{claim}{Claim}

\headers{Root-convergence GMRES(1)}{Y. He}

\title{The worst-case root-convergence factor of GMRES(1) }

\author{Yunhui He\thanks{Department of Mathematics, University of Houston, 3551 Cullen Blvd, Room 641, Houston, Texas 77204-3008, USA (\email{yhe43@central.uh.edu}).}}

\usepackage{amsopn}

\ifpdf
\hypersetup{
  pdftitle={The worst-case root-convergence factor of GMRES(1)},
  pdfauthor={Yunhui He}
}
\fi

\begin{document}

\maketitle

\begin{abstract}
	In this work, we analyze the asymptotic convergence factor of minimal residual iteration (MRI) (or GMRES(1)) for solving linear systems $Ax=b$  based on  vector-dependent nonlinear eigenvalue problems. The worst-case root-convergence factor is derived for linear systems with $A$ being symmetric or $I-A$ being skew-symmetric.  When $A$ is symmetric, the asymptotic convergence factor highly depends on the initial guess.   While $M=I-A$ is skew-symmetric,  GMRES(1) converges unconditionally and the worst-case root-convergence factor relies solely on the spectral radius of $M$. We also derive the q-linear convergence factor, which is the same as the worst-case root-convergence factor. Numerical experiments are presented to validate our theoretical results.

\end{abstract}


	\section{Introduction} \label{sec:intro}
	In this work, we consider solving the following linear system
	\begin{equation}\label{eq:linear-sys}
		Ax=b,
	\end{equation}
	where $A \in \mathbb{R}^{n\times n}$ is invertible. One of the many iterative methods for solving \eqref{eq:linear-sys} is the Generalized Minimal Residual Method (GMRES) \cite{saad1986gmres}.  In the literature, there is a substantial amount of work on the convergence analysis of GMRES, see \cite{greenbaum1994gmres, embree2022descriptive, toh1997gmres,van1993superlinear,robbe2002convergence,liesen2004convergence,liesen2000computable,titley2014gmres,liesen2004worst,arioli2009analysis}. The convergence behavior of GMRES is quite complex. For example, \cite{greenbaum1996any} shows that any nonincreasing convergence curve can be obtained with GMRES applied to a matrix having any desired eigenvalues.  Many variants of GMRES have been developed, such as GMRES with deflated restarting \cite{morgan2002gmres,giraud2010flexible,lin2012simpler, agullo2014block}, restarted GMRES (GMRES($m$)) \cite{morgan1995restarted,baker2005technique,frommer1998restarted,vecharynski2010cycle}, nested GMRES methods \cite{van1994gmresr}, hybrid GMRES \cite{nachtigal1992hybrid,simoncini1996hybrid},  and weighted GMRES \cite{guttel2014some,essai1998weighted}.
	
	There have been many efforts on the convergence analysis of restarted GMRES \cite{joubert1994convergence,faber2013properties,zitko2000generalization, zhong2008complementary},  which selects a relatively large restart parameter $m$ to imitate the full GMRES process. However, there still lacks better understanding of restarted GMRES, even for GMRES(1)(or Minimal Residual Iteration), see Algorithm \ref{alg:MRI}. Some numerical results show that using small values of the restart parameter yields convergence in fewer iterations compared to larger values. For example,  numerical experiments in \cite{embree2003tortoise} show that GMRES(1) for some linear systems with dimension three converges exactly in three iterations, while GMRES(2) stagnates. Moreover,  the convergence of GMRES(1) can be extremely sensitive to the initial guess.  There is still much to understand about the convergence behavior of GMRES($m$), especially regarding how the asymptotic convergence factor is sensitive to the initial guess. Improved GMRES convergence theory is needed to offer practical guidance on using this algorithm. We limit ourselves to investigating GMRES(1). In \cite[Proposition 2.1]{saad2000further}, the author derived the following upper bound for GMRES(1) when $A$ is symmetric definite:
	\begin{equation}
		\|r_{k+1}\| \leq  \left|\frac{\lambda_{\max}(A)-\lambda_{\min}(A)}{\lambda_{\max}(A)+\lambda_{\min}(A)}\right| \|r_k\|.
	\end{equation}
	However, there is not much investigation on the above bound. For different initial errors $r_0$, the upper bound might not be tight. What is the asymptotic convergence factor for GMRES(1)? Is it possible that there exists $r_0$ such that  the above bound is achieved for every $k$? In this work, we address these questions and provide additional results on how the initial guess or initial error affects the asymptotic convergence factor.
	
	\begin{algorithm}
		\caption{Minimal Residual Iteration (MRI) or  GMRES(1)} \label{alg:MRI}
		\begin{algorithmic}[1] 
			\STATE Given $x_0$ and set $r_0=Ax_0-b$
			\STATE  For $k=0, 1, \cdots$ until convergence Do:
			\begin{itemize}
				\item compute $\alpha_k=\frac{r_k^TA^Tr_k}{r_k^TA^TAr_k}$
				\item compute $x_{k+1}=x_k-\alpha_k r_k$
				\item compute  $r_{k+1}=r_k-\alpha_kAr_k$
			\end{itemize}
			EndDo
		\end{algorithmic}
	\end{algorithm}
	
	\begin{definition}[Root  convergence]
		Assume that a sequence $\{y_k\}$ generated by an iterative method converges to $y^*$. We define the root-convergence factor of the iterative method as 
		\begin{equation}\label{eq:root-convergence-initial}
			\varrho(y_0) =\lim\sup_{k\rightarrow \infty} \varrho_k(y_0), \quad \text{where}\quad \varrho_k(y_0)=\|y_k-y^*\|^{1/k},
		\end{equation}
		where $y_0$ is the initial guess for the iterative method. Further more, we call
		
		\begin{equation}\label{eq:worst-case-convergence}
			\varrho^* =\max_{y_0} \varrho(y_0),
		\end{equation}
		the worst-case root-convergence factor of the iterative method.
	\end{definition}
	
	\begin{definition}[q-linear convergence,\cite{toth2015convergence}] 
		We call a sequence $\{y_k\}$ converges q-linearly with q-factor $\sigma \in(0,1)$ to $y^*$ if
		\begin{equation}\label{eq:q-convergence}
			\|y_k-y^*\| \leq \sigma\|y_{k-1}-y^*\|, \quad k\geq 1.
		\end{equation}
	\end{definition}
	
	Motivated by \cite{both2019anderson} and \cite{de2024convergence}, we analyze the convergence behavior of GMRES(1) based on eigenvector-dependent eigenvalue problems. The main goal of this work is to offer theoretical results of \eqref{eq:root-convergence-initial}, \eqref{eq:worst-case-convergence}, and \eqref{eq:q-convergence} for GMRES(1). The main contributions of this work are summarized as follows.
	We provide a clear picture of the eigenpairs of vector-dependent eigenvalue problems associated with the GMRES(1) iterations. We derive the worst-case root-convergence factor of GMRES(1) with $A$ being symmetric and $M=I-A$ being skew-symmetric. 
	Our results also show that how the initial guess can affect the asymptotic convergence factor.  For some special initial guesses for GMRES(1), we present the corresponding root-convergence factors. 
	We add some theoretical results for rAA(1).  We provide counterexamples for a conjecture of restarted Anderson acceleration with window size one, denoted as rAA(1), presented in \cite{de2024convergence}.
	A comparison between GMRES(1) and rAA(1) is discussed. We see that for $M$ being skew-symmetric, GMRES(1) always converges with the worst-case convergence factor depending only on the spectral radius of $M$ and has a faster convergence factor compared to rAA(1). For $A$ being symmetric definite, GMRES(1) always converges, and the worse-case root-convergence factor is $\left|\frac{\lambda_{\max}(A)-\lambda_{\min}(A)}{\lambda_{\max}(A)+\lambda_{\min}(A)}\right|$ and is one for indefinite system. We also derive the q-linear convergence factor, which is the same as the worst-case root-convergence factor.
	These results can help us better understand GMRES(1) and rAA(1) methods. They have the potential to offer guidance on designing faster algorithms, such as preconditioning.
	
	The rest of this work is organized as follows. In section \ref{sec:symmetric}, we analyze GMRES(1) for $A$ being symmetric and derive the worst-case root-convergence factor. Some new results of rAA(1) have been presented and we provide counterexamples of a conjecture from \cite{de2024convergence}.  In section \ref{sec:skew}, we consider $M=I-A$ being skew-symmetric, where the worst-case convergence factor is proposed. GMRES(1) always converges for this case. Then, we briefly discuss the q-linear convergence of GMRES(1). Numerical experiments are presented in section \ref{sec:num} to validate our theoretical results. Finally, we draw conclusions in section \ref{sec:con}.
	
	\section{Symmetric case for GMRES(1)}\label{sec:symmetric}
	From Algorithm \ref{alg:MRI}, we have
	\begin{equation}\label{eq:residual-form}
		r_{k+1}=(I -\alpha_k(r_k) A)r_k.
	\end{equation}
	Let us consider the following mapping $\Phi$:
	\begin{equation*}
		\Phi(v)=(I -\alpha(v) A)v,
	\end{equation*}
	where
	\begin{equation*}
		\alpha(v)=\frac{\langle v, Av \rangle}{\langle Av,Av \rangle}=\frac{\langle Av, S(Av)\rangle}{\langle Av,Av\rangle},
	\end{equation*}
	where $S=\frac{1}{2}(A^{-1}+A^{-T})$.
	Then \eqref{eq:residual-form} can be written as  $r_{k+1}=\Phi(r_k)$.

	It is known that for an iterative function $\mathcal{F}$ if it is differentiable at $x^*$ (a fixed point of  $\mathcal{F}$), then the worst-case root-convergence factor is given by the spectral radius of $\mathcal{F}'(x^*)$, see  \cite{ortega2000iterative}.
	Here, we cannot employ this approach to analyze GMRES(1).

	\subsection{One step of GMRES(1)}
	Define the following vector-dependent matrix
	\begin{equation*}
		\mathcal{I}(v) =I -\alpha(v) A, \quad  v\neq 0 .
	\end{equation*}
	It can be seen that $\mathcal{I}(v)v=\Phi(v)$. Based on this, every one step of the residuals of GMRES(1) satisfy
	\begin{equation*}
		r_{k+1} =\mathcal{I}(r_{k}) r_{k}, \quad k=0, 1, \cdots
	\end{equation*} 
	To understand the worst-case root-convergence factor of GMRES(1), we first investigate 
	\begin{equation*}
		\max_{v\in\mathbb{R}^n}\frac{\|\mathcal{I}(v)v\|}{\|v||}.
	\end{equation*}
	We consider the following eigenvector-dependent nonlinear eigenvalue problem: 
	\begin{equation}\label{eq:GEVPI}
		\mathcal{I}_2(u)u=\lambda u, \quad \text{where} \quad \mathcal{I}_2(u)= (I -\alpha(u)A^T)(I -\alpha(u)A).
	\end{equation}
	For simplicity, when considering $A$ as symmetric, we assume that the following condition is satisfied.
	\begin{assumption}
		Suppose that  $A\in \mathbb{R}^{n \times n}$ is symmetric with $p$ distinguish eigenvalues, denoted as  $\{a_i\}_{i=1}^p$. For each $i$, there are $n_i$ linearly independent orthonormal eigenvectors $\{u_{i,j}\}_{j=1}^{n_i}$, where $\sum_{i=1}^{p} n_i=n$.
	\end{assumption}
	
	\begin{lemma}\label{lem:RQ-max}
		\begin{equation}\label{eq:vector-matrix-norm}
			\max_{v\neq 0} \frac{\|\mathcal{I}(v)v\|}{\|v\|}=\Gamma^*,
		\end{equation}
		where  $\Gamma^*=\max_{\lambda}\sqrt{|\lambda|}$ with $\lambda$ being the eigenvalues of \eqref{eq:GEVPI}.
	\end{lemma}
	\begin{proof}
		Since $\|\mathcal{I}(v)\|=\min_{c \in \mathbb{R}}\|(I-cA)v\|$, \eqref{eq:vector-matrix-norm} is a direct conclusion from \cite[Theorems 3.5, 3.6]{faber2013properties}.
	\end{proof}
	In the following, we focus on the case where $A$ is symmetric.
	\begin{lemma}
		Let $z=\sum_{i=1}^{n}y_i$, where $y_i$ are orthogonal eigenvectors of the symmetric matrix $A\in \mathbb{R}^{n\times n}$ associated with eigenvalues $s_i$. Then,
		\begin{equation}\label{eq:RQ-prop-I}
			\alpha(z)=\frac{\langle z,Az\rangle}{\langle Az,Az\rangle}=\frac{\sum_{i=1}^{n} s_i \|y_i\|^2}{ \sum_{i=1}^{n} s_i^2 \|y_i\|^2}.
		\end{equation}
	\end{lemma}
	For convenience, we consider the orthonormal eigenvectors of $A$ in the following.
	
	\begin{lemma}\label{lem:eigpair-two-condition}
		Let $u_i,u_j$ be two orthonormal eigenvectors of the symmetric matrix $A$ with eigenvalues $a_i$ and $a_j\neq a_i$, respectively.  Consider $z_{ij}=c_i u_i+c_j u_j$ with $\epsilon=\frac{c_j}{c_i}$. Then, $z_{ij}$ is an eigenvector of \eqref{eq:GEVPI} associated with non-unit eigenvalue if and only if  $a_ia_j>0$ and $\epsilon^2=\frac{a_i}{a_j}$, and the corresponding eigenvalue is
		\begin{equation*}
			\lambda_{ij} =\frac{(a_j-a_i)^2}{(a_j+a_i)^2}.
		\end{equation*}
	\end{lemma}
	
	\begin{proof}
		By standard calculation, we have
		\begin{equation}\label{eq:form-alphazij}
			\alpha(z_{ij}) =\frac{a_i+a_j\epsilon^2}{a_i^2+a_j^2\epsilon^2}.
		\end{equation}
		Assume that 
		\begin{equation*}
			\mathcal{I}_2(z_{ij})(z_{ij}) =(1-\alpha(z_{ij})a_i)^2 c_iu_i + (1-\alpha(z_{ij})a_j)^2 c_ju_j=\lambda_{ij} z_{ij}.
		\end{equation*}
		It follows that
		\begin{equation*}
			(1-\alpha(z_{ij}) a_i)^2 c_iu_i =\lambda_{ij}c_i u_i \quad \text{and} \quad 
			(1-\alpha(z_{ij}) a_j)^2 c_ju_j =\lambda_{ij}c_j u_j.
		\end{equation*}
		It is equivalent to $(1-\alpha(z_{ij}) a_i)^2 =	(1-\alpha(z_{ij}) a_j)^2$, that is, $2=(a_j+a_i)\alpha(z_{ij})$. 
		Using $\alpha(z_{ij})$ defined in \eqref{eq:form-alphazij}, we obtain $\epsilon^2=\frac{a_i}{a_j}$.
		Next, we compute the eigenvalue
		\begin{equation*}
			\lambda_{ij}=(1- \alpha(z_{ij})a_i)^2=\left(\frac{a_j(a_j-a_i)\epsilon^2}{a_i^2+a_j^2\epsilon^2}\right)^2=\frac{ (a_j-a_i)^2}{(a_j+a_i)^2},
		\end{equation*}
		which completes the proof.
	\end{proof}
	
	\begin{theorem}\label{thm:multiple-two-eigpair-I} 
		Assume that $A$ is symmetric. Let  $i_1, i_2 \in \{1,2,\cdots, p\}$ and $i_1\neq i_2$. Consider
		\begin{equation*} 
			u^*= \sum_{s=1}^{n_{i_1}} c_{i_1,s} u_{i_1,s} +  \sum_{r=1}^{n_{i_2}} c_{i_2,r} u_{i_2,r}, \quad  c_{i_1,1}c_{i_2,1}\neq 0,
		\end{equation*}
		and $u_{i_1 i_2}=   c_{i_1} u_{i_1,1} + c_{i_2} u_{i_2,1}$,	where	
		\begin{equation} \label{eq:coeff-square-require-I}
			c_{i_1}^2=\sum_{s=1}^{n_{i_1}}c_{i_1,s}^2, \quad c_{i_2}^2=\sum_{r=1}^{n_{i_1}}c_{i_1,r}^2,
		\end{equation}
		Then,
		\begin{equation}\label{eq:invariant-RQ1-I}
			\alpha(u^*)=\alpha(u_{i_1 i_2}). 
		\end{equation}
		Moreover,  if $\left(\frac{c_{i_2}}{c_{i_1}}\right)^2=\frac{a_{i_1}}{a_{j_1}}>0$, then $u^*$ and $u_{i_1 i_2}$ are eigenvectors of \eqref{eq:GEVPPi} associated with the same eigenvalue, i.e.,
		$\mathcal{I}_2(u^*)u^*=\lambda_{i_1 i_2} u^*$ and  $\mathcal{I}_2(u_{i_1 i_2})u_{i_1 i_2}=\lambda_{i_1 i_2} u_{i_1 i_2}$,
		where
		\begin{equation}\label{eq:eig-asso2-I}
			\lambda_{i_1 i_2}(\epsilon_{i_1 i_2}) =\frac{(a_{i_2}-a_{i_1})^2}{(a_{i_1}+a_{i_2})^2}.
		\end{equation}
	\end{theorem}
	\begin{proof}
		When $c_{i_1}=0$ or $c_{i_2}=0$, it is obvious that $u^*$ and $u_{i_1i_2}$ are eigenvectors of \eqref{eq:GEVPPi} and $\lambda=0$. Next, we prove \eqref{eq:invariant-RQ1-I}. From \eqref{eq:RQ-prop-I}, we have
		\begin{equation*}
			\alpha(u^*)= \frac{\sum_{s=1}^{n_{i_1}}a_{i_1}c_{i_1,s}^2 + \sum_{r=1}^{n_{i_2}}a_{i_2}c_{i_1,r}^2 }{ \sum_{s=1}^{n_{i_1}}a_{i_1}^2c_{i_1,s}^2 +\sum_{r=1}^{n_{i_2}}a_{i_2}^2c_{i_1,r}^2 }= \frac{a_{i_1} \sum_{s=1}^{n_{i_1}}c_{i_1,s}^2 + a_{i_2}\sum_{r=1}^{n_{i_2}}c_{i_1,r}^2 }{ a_{i_2}^2\sum_{s=1}^{n_{i_1}} c_{i_1,s}^2 +a_{i_2}^2\sum_{r=1}^{n_{i_2}}c_{i_1,r}^2 },
		\end{equation*}
		and 
		\begin{equation*}
			\alpha(u_{i_1 i_2}) =\frac{a_{i_1}c_{i_1}^2 + a_{i_2}c_{i_2}^2}{a_{i_1}^2c_{i_1}^2 +a_{i_2}^2c_{i_2}^2 }.
		\end{equation*}
		Using conditions in \eqref{eq:coeff-square-require-I} gives $\alpha(u^*)=\alpha(u_{i_1 i_2})$.
		
		From Lemma \ref{lem:eigpair-two-condition}, we know that $(\lambda_{i_1 i_2}(\epsilon_{i_1 i_2}), u_{i_1 i_2})$ is eigenpair of \eqref{eq:GEVPI}. For simplicity, let $\beta=\alpha(u^*)$. We have
		\begin{equation*}
			(1-\beta a_{i_1})^2 c_{i_1}u_{i_1,1}=\lambda_{i_1 i_2}c_{i_1}u_{i_1,1} \quad  \text{and}\quad 
			(1-\beta a_{i_2})^2 c_{i_2}u_{i_2,1}=\lambda_{i_1 i_2}c_{i_2}u_{i_2,1}.	
		\end{equation*}
		It means
			$(1-\beta a_{i_1})^2   =\lambda_{i_1 i_2}$ and  
			$(1-\beta a_{i_2})^2  =\lambda_{i_1 i_2}$.	
		Thus,
		\begin{align*}
			\mathcal{I}_2(u^*)u^*&=\sum_{s=1}^{n_{i_1}} (1-\beta a_{i_1})^2 c_{i_1,s}u_{i_1,s}+\sum_{r=1}^{n_{i_2}} (1-\beta a_{i_2})^2c_{i_2,r}u_{i_2,r}\\
			&=  \sum_{s=1}^{n_{i_1}} \lambda_{i_1 i_2} c_{i_1,s}u_{i_1,s}+\sum_{r=1}^{n_{i_2}} \lambda_{i_1 i_2} c_{i_2,r}u_{i_2,r}\\
			&= \lambda_{i_1 i_2}\left(\sum_{s=1}^{n_{i_1}}  c_{i_1,s}u_{i_1,s}+\sum_{r=1}^{n_{i_2}}  c_{i_2,r}u_{i_2,r}\right)\\
			&=\lambda_{i_1 i_2} u^*,
		\end{align*}
		which completes the proof.
	\end{proof}

	\begin{theorem}\label{thm:all-eigv-form-I}
		Suppose that  $A$ is symmetric. If $A$ is definite, then all the eigenvectors of \eqref{eq:GEVPPi} associated with nonzero eigenvalue are given by
		\begin{equation}\label{eq:all-eig-form-I}
			u^*= \sum_{s=1}^{n_{i_1}} c_{i_1,s} u_{i_1,s} +  \sum_{r=1}^{n_{i_2}} c_{i_2,r} u_{i_2,r}, \quad \text{where}\quad i_1\neq i_2, i_1, i_2 \in \{1,2,\cdots, p\},
		\end{equation}
		where at least one term in $\{c_{i_1,s}\}_{s=1}^{n_{i_1}}$ is nonzero, at least one term in $\{c_{i_2,r}\}_{r=0}^{n_{i_2}}$ is nonzero, and $\frac{c_{i_2}^2}{c_{i_1}^2}=\frac{a_{i_1}}{a_{i_2}}>0$ with $\sum_{s=1}^{n_{i_1}}c_{i_1,s}^2$ and $c_{i_2}^2=\sum_{r=1}^{n_{i_1}}c_{i_1,r}^2$.

		If $A$ is indefinite, then all the eigenvectors of \eqref{eq:GEVPPi}  associated with non-unit eigenvalues are given by \eqref{eq:all-eig-form-I}, and the eigenvectors associated with eigenvalue one are given by $u^*$, where $\langle Au^*,u^*\rangle=0$.
	\end{theorem}

	Formula \eqref{eq:all-eig-form-I} means that the eigenvectors associated with nonzero eigenvalue of \eqref{eq:GEVPPi} are any linear combination of eigenvectors of $A$ associated with exactly any two distinguish eigenvalues of $A$.
	
	\begin{proof}
		For any $i \in\{1,2,\cdots, p\}$, let $u_*=\sum_{s}^{n_i}c_{i,s}u_{i,s}\neq 0$. Then, $(I -\alpha(u_*) A) u_*=0$, which means $\mathcal{I}_2(u_*)u_*=0$ and $u_*$ is an eigenvector corresponding to zero eigenvalue. 
		
		Assume that the eigenvector of \eqref{eq:GEVPI} is given by
		\begin{equation}\label{eq:nonzero-ci-I}
			u^*=\sum_{i=1}^p  \sum_{s=1}^{n_{i}} c_{i,s} u_{i,s},
		\end{equation}
		which contains at least two different integers $i,j$ such that $c_{i,s} c_{j,t}\neq 0$. 
		
		When $A$ is indefinite, there exists a vector $v$ such that $\langle Av,v\rangle=0$. Then $\alpha(v)=0$ and $\mathcal{I}_2(v)(v)=v$. In this situation, $\lambda=1$ is an eigenvalue of \eqref{eq:GEVPI}.

		Next, we prove that in \eqref{eq:nonzero-ci-I}, there are only two different integers $i,j$ such that $c_{i,s} c_{j,t}\neq 0$.  For simplicity, let $\beta=\alpha(u^*)$.  Assume $\mathcal{I}_2(u^*)u^*= \lambda  u^*, \lambda\neq 0$. Then, we have
		\begin{equation}\label{eq:component-wise-term-I}
			c_{i,s}(1-\beta a_i)^2u_{i,s}=\lambda c_{i,s}u_{i,s}, \quad  (i, s) \in\{ 1,2,\cdots, p\} \times \{1,2,\cdots, n_i\}.  
		\end{equation}
		If $\beta=0$, then $\langle Au^*,u^* \rangle$ and $\lambda=1$. Consider $\beta\neq 0$.  Define the quadratic function $g(b)=(1-\beta b)^2$.  For a given value $g^*$, there are at most two different values $b_1,b_2$ such that $g(b_1)=g(b_2)=g^*$. To guarantee \eqref{eq:component-wise-term-I},  there are at most two different integers $i, j$ such that $c_{i,s} c_{j,t}\neq 0$, where $s\in \{1,2, \cdots, n_i\}$ and $t\in \{1,2, \cdots, n_j\}$, and the coefficients conditions are given in Theorem \ref{thm:multiple-two-eigpair-I}.
		Thus, all the eigenvectors of \eqref{eq:GEVPPi} associated with non-unit eigenvalues are any linear combination of eigenvectors of $A$ with exactly two distinguish eigenvalues of $A$.
	\end{proof}
	From the above discussion, we know that 
	\begin{equation}\label{eq:def-Gamma}
		\Gamma^* =
		\begin{cases}
			\max_{i\neq j, i,j\in \{1,2, \cdots,p\}} \frac{|a_j-a_i|}{|a_j|+|a_i|}, & \text{if $A$ is definite}\\
			1, & \text{if $A$ is indefinite}
		\end{cases}
	\end{equation}

	\begin{theorem}\label{thm:GMRES1-worst-case}
		For GMRES(1) applied to $Ax=b$, where $A$ is symmetric, the worst-case root-convergence factor  is 
		\begin{equation}\label{eq:gam1-form}
			\varrho^*_{GMRES(1)}\leq 
			\begin{cases} 
				\left|\frac{\lambda_{\max}(A)- \lambda_{\min}(A)}{\lambda_{\max}(A)+\lambda_{\min}(A)}\right|, & \text{if $A$ is definite}\\
				1, & \text{if $A$ is indefinite}
			\end{cases}
		\end{equation}
	\end{theorem}
	\begin{proof}
		From Lemma \ref{lem:RQ-max}, we have  $\|r_k\|\leq \Gamma^*\|r_{k-1}\|$. Thus, $\|r_k\|\leq (\Gamma^*)^k\|r_0\|$ and $\varrho^*\leq\Gamma^*$.
	\end{proof}
	From the above result, we know that when $A$ is  symmetric definite, then the worst-case root-convergence factor is not worse than $\left|\frac{\lambda_{\max}(A)- \lambda_{\min}(A)}{\lambda_{\max}(A)+\lambda_{\min}(A)}\right|$. When $A$ is symmetric indefinite,  $\varrho^*=1$. In the following, we will show that there exists $r_0$ such that the upper bound can be achieved. 
	
	\subsection{Two steps of GMRES(1)}
	
	To better understand how the initial guess affects the root-convergence factor of GMRES(1), we consider two steps of GMRES(1), and define the following vector-dependent matrix
	\begin{equation*}
		\Pi(v) =(I -\alpha(\Phi(v)) A)(I -\alpha(v) A), \quad  v\neq 0 .
	\end{equation*}
	It can be seen that $\Pi(v)v=\Phi(\Phi(v))$. Based on this, every two steps of the residuals of GMRES(1) satisfy
	\begin{equation*}
		r_{2(k+1)} =\Pi(r_{2k}) r_{2k}, \quad k=0, 1, \cdots
	\end{equation*} 
	The above iteration is related to the following eigenvector-dependent nonlinear eigenvalue problem 
	\begin{equation}\label{eq:GEVPPi}
		\Pi(u)u=\lambda u.
	\end{equation}

	\begin{theorem}\label{thm:two-eigenvectors}
		Let $u_i,u_j$ be two orthonormal eigenvectors of the symmetric matrix $A$ with eigenvalues $a_i$ and $a_j\neq a_i$, respectively. Define
		\begin{equation*}
			z_{ij}=c_i u_i+c_ju_j=c_i(u_i+\epsilon u_j), \quad \epsilon=\frac{c_j}{c_i}, \quad c_i c_j\neq 0.
		\end{equation*}
		Then
		\begin{equation*}
			\Pi(z_{ij})z_{ij}=\lambda_{ij}(\epsilon) z_{ij},   
		\end{equation*}
		where
		\begin{equation*}
			\lambda_{ij}(\epsilon) =\frac{(a_j-a_i)^2}{a_i^2+\epsilon^2a_j^2}\frac{1}{1+\frac{1}{\epsilon^2}}.
		\end{equation*}
	\end{theorem}
	\begin{proof}
		This can be easily derived following  \cite{de2024convergence}. Thus, we omit it here. 
		
	\end{proof}

	\begin{lemma}\label{lem:maxim-form}
		The maximum value of $\lambda_{ij}(\epsilon)$ is given by
		\begin{equation}\label{eq:largest-eig-G}
			\max_{\epsilon \in \mathbb{R}}\lambda_{ij}(\epsilon)=\lambda_{ij} \left(\pm\sqrt{\frac{|a_i|}{|a_j|}}\right)=\frac{(a_j-a_i)^2}{(|a_j|+|a_i|)^2}\leq 1.
		\end{equation}
	\end{lemma}
	\begin{proof}
		To find the maximum of $\lambda_{ij}(\epsilon)$, we only need to find the minimum of 
		$$g(\tau)=(a_i^2+\tau a_j^2) (1+\frac{1}{\tau}),$$
		where $\tau=\epsilon^2>0$. In fact, $g'(\tau)=a_j^2(1+\frac{1}{\tau})+(a_i^2+\tau a_j^2) (-\frac{1}{\tau^2})$. Let $g'(\tau)=0$. We obtain $\tau^2=\epsilon^4=\frac{a_i^2}{a_j^2}$. We notice that when $\epsilon \rightarrow 0$, $g(\tau)\rightarrow \infty$. Thus, when $\epsilon=\pm \sqrt{\frac{|a_i|}{|a_j|}}$, $g(\tau)$ obtains the minimum. 
	\end{proof}
	
	From Theorem \ref{thm:two-eigenvectors} and Lemma \ref{lem:maxim-form}, we see that the largest eigenvalue \eqref{eq:largest-eig-G} is the same as that of \eqref{eq:GEVPI}. In fact, we have the following result.
	\begin{corollary}
		If  $(\lambda, u)$ is  an eigenpair of \eqref{eq:GEVPI}, then $(\lambda, u)$ is an eigenpair of \eqref{eq:GEVPPi}.
	\end{corollary}
	
	We explain the above corollary in more details. If $(\lambda, z_{ij})$ is  an eigenpair of \eqref{eq:GEVPI}, where $z_{ij}=c_i u_i+c_j u_j$ with $\epsilon=\frac{c_j}{c_i}$, $a_ia_j>0$ and $\epsilon^2=\frac{a_i}{a_j}$, one can show that $\alpha(\Phi(z_{ij})) =\alpha( z_{ij})$. Then, it follows that $\Pi(z_{ij})z_{ij}=(I -\alpha(\Phi(z_{ij}))A)z_{ij})(I-\alpha(z_{ij})A)z_{ij}=\lambda z_{ij}$. For general eigenpairs of \eqref{eq:GEVPI}, we can use the properties in the proof of Theorem \ref{thm:multiple-two-eigpair-I} to process the proof here.

	\begin{theorem}\label{thm:multiple-two-eigpair} 
		Suppose that  $A\in \mathbb{R}^{n \times n}$ is symmetric. Let  $i_1, i_2 \in \{1,2,\cdots, p\}$ and $i_1\neq i_2$. Consider
		\begin{equation*} 
			u^*= \sum_{s=1}^{n_{i_1}} c_{i_1,s} u_{i_1,s} +  \sum_{r=1}^{n_{i_2}} c_{i_2,r} u_{i_2,r}, \quad  c_{i_1,1}c_{i_2,1}\neq 0,
		\end{equation*}
		and $u_{i_1 i_2}=   c_{i_1} u_{i_1,1} + c_{i_2} u_{i_2,1}$,
		where	
		\begin{equation} \label{eq:coeff-square-require}
			c_{i_1}^2=\sum_{s=1}^{n_{i_1}}c_{i_1,s}^2, \quad c_{i_2}^2=\sum_{r=1}^{n_{i_1}}c_{i_1,r}^2.
		\end{equation}
		Then, 
		\begin{equation*} 
			\Pi(u^*)u^*=\lambda_{i_1 i_2} u^*, \quad \Pi(u_{i_1 i_2})u_{i_1 i_2}=\lambda_{i_1 i_2} u_{i_1 i_2},
		\end{equation*}
		where 
		\begin{equation}\label{eq:eig-asso2}
			\lambda_{i_1 i_2}(\epsilon_{i_1 i_2}) =\frac{(a_{i_2}-a_{i_1})^2}{a_{i_1}^2+\epsilon^2_{i_1,i_2}a_{i_2}^2}\frac{1}{1+\frac{1}{\epsilon_{i_1 i_2}^2}}, \quad \epsilon_{i_1 i_2} =\frac{c_{i_2}}{c_{i_1}}.
		\end{equation}
	\end{theorem}
	\begin{proof}
		From Theorem \ref{thm:multiple-two-eigpair-I},  we have $\alpha(u^*)=\alpha(u_{i_1 i_2})$.
		Consider
		\begin{equation*}
			u^*_1 = (I -\alpha(u^*) A)u^*=\sum_{s=1}^{n_{i_1}} (1-\alpha(u^*)a_{i_1}) c_{i_1,s} u_{i_1,s} + \sum_{r=1}^{n_{i_2}} (1-\alpha(u^*)a_{i_2}) c_{i_2,r} u_{i_2,r},
		\end{equation*}
		and 
		\begin{equation*}
			u_{i_1 i_2,1}=(I -\alpha(u_{i_1 i_2}) A)u_{i_1 i_2}=(1-\alpha(u^*)a_{i_1}) c_{i_1} u_{i_1,1} +  (1-\alpha(u^*)a_{i_2}) c_{i_2}u_{i_2,1}.
		\end{equation*}
		We have
		\begin{equation*}
			\sum_{s=1}^{n_{i_1}} (1-\alpha(u^*)a_{i_1})^2 c_{i_1,s}^2 + \sum_{r=1}^{n_{i_2}} (1-\alpha(u^*)a_{i_2})^2 c_{i_2,r}^2= 
			(1-\alpha(u^*)a_{i_1})^2 c_{i_1}^2+  (1-\alpha(u^*)a_{i_2})^2 c_{i_2}^2.
		\end{equation*}
		It follows that $\alpha((I -\alpha(u^*) A)u^*)=\alpha((I -\alpha(u_{i_1 i_2}) A)u_{i_1 i_2})$.
		
		From Theorem \ref{thm:two-eigenvectors}, we know that $(\lambda_{i_1 i_2}(\epsilon_{i_1 i_2}), u_{i_1 i_2})$ is eigenpair of \eqref{eq:GEVPPi}. For simplicity, let $\beta_2=\alpha(u^*_1)$ and $\beta_1=\alpha(u^*)$. We have
		\begin{align*}
			(1-\beta_2 a_{i_1})(1-\beta_1a_{i_1}) c_{i_1}u_{i_1,1}&=\lambda_{i_1 i_2}c_{i_1}u_{i_1,1}, \\
			(1-\beta_2 a_{i_2})(1-\beta_1a_{i_2}) c_{i_2}u_{i_2,1}&=\lambda_{i_1 i_2}c_{i_2}u_{i_2,1}.	
		\end{align*}
		It means
		\begin{equation*}
			(1-\beta_2 a_{i_1})(1-\beta_1a_{i_1})  =\lambda_{i_1 i_2} \quad \text{and} \quad 
			(1-\beta_2 a_{i_2})(1-\beta_1a_{i_2})  =\lambda_{i_1 i_2}.	
		\end{equation*}
		Thus,
		\begin{align*}
			\Pi(u^*)u^*&=\sum_{s=1}^{n_{i_1}} (1-\beta_2 a_{i_1})(1-\beta_1a_{i_1}) c_{i_1,s}u_{i_1,s}+\sum_{r=1}^{n_{i_2}} (1-\beta_2 a_{i_2})(1-\beta_1a_{i_2}) c_{i_2,r}u_{i_2,r}\\
			&=  \sum_{s=1}^{n_{i_1}} \lambda_{i_1 i_2} c_{i_1,s}u_{i_1,s}+\sum_{r=1}^{n_{i_2}} \lambda_{i_1 i_2} c_{i_2,r}u_{i_2,r}\\
			&= \lambda_{i_1 i_2}\left(\sum_{s=1}^{n_{i_1}}  c_{i_1,s}u_{i_1,s}+\sum_{r=1}^{n_{i_2}}  c_{i_2,r}u_{i_2,r}\right)\\
			&=\lambda_{i_1 i_2} u^*,
		\end{align*}
		which completes the proof.
	\end{proof}

	\begin{theorem}\label{thm:all-eigv-form}
		Suppose that $A$ is symmetric. If $A$ is definite, then all the eigenvectors of \eqref{eq:GEVPPi} associated with nonzero eigenvalue are given by
		\begin{equation}\label{eq:all-eig-form}
			u^*= \sum_{s=1}^{n_{i_1}} c_{i_1,s} u_{i_1,s} +  \sum_{r=1}^{n_{i_2}} c_{i_2,r} u_{i_2,r}, \quad \text{where}\quad i_1\neq i_2, i_1, i_2 \in \{1,2,\cdots, p\},
		\end{equation}
		where at least one term in $\{c_{i_1,s}\}_{s=1}^{n_{i_1}}$ is nonzero and at least one term in $\{c_{i_2,r}\}_{r=0}^{n_{i_2}}$ is nonzero. Moreover, all the eigenvalues are less than one and are given by \eqref{eq:eig-asso2}.
		
		If $A$ is indefinite, then all the eigenvectors of \eqref{eq:GEVPPi} associated with non-unit eigenvalues are given by \eqref{eq:all-eig-form}, and the eigenvectors associated with eigenvalue one are given by $u^*$, where $\langle Au^*,u^* \rangle=0$.
	\end{theorem}
	
	\begin{proof}
		For any $i \in\{1,2,\cdots, p\}$, let $u_*=\sum_{s}^{n_i}c_{i,s}u_{i,s}\neq 0$. Then, $(I -\alpha(u_*) A) u_*=0$.  So, such $u_*$ cannot be eigenvectors of \eqref{eq:GEVPPi} ($\alpha(0)$ is not defined). 
		
		Assume that the eigenvector of \eqref{eq:GEVPPi} is given by $u^*=\sum_{i=1}^p  \sum_{s=1}^{n_{i}} c_{i,s} u_{i,s}$,
		which contains at least two different integers $i,j$ such that $c_{i,s} c_{j,t}\neq 0$. 
		
		When $A$ is indefinite, there exists a vector $v$ such that $\langle Av,v\rangle=0$. Then $\alpha(v)=0$ and $\Pi(v)=v$. In this situation, $\lambda=1$ is an eigenvalue of \eqref{eq:GEVPPi}.

		Next, we prove that in $u^*$, there are at most two different integers $i,j$ such that $c_{i,s} c_{j,t}\neq 0$.  For simplicity, let $\beta_1=\alpha(u^*)$ and $\beta_2=\alpha(\Phi(u^*))=\alpha((I -\alpha(u^*) A)u^*)$. 
		
		Using $\Pi (u^*)u^*= \lambda  u^*, \lambda\neq 0$, we have
		\begin{equation}\label{eq:component-wise-term}
			c_{i,s}(1-\beta_2 a_i) (1-\beta_1 a_i)u_{i,s}=\lambda c_{i,s}u_{i,s}, \quad  (i, s) \in\{ 1,2,\cdots, p\} \times \{1,2,\cdots, n_i\}.  
		\end{equation}
		If $\beta_1=0$, then $\beta_2=0$. If $\beta_1\beta_2\neq 0$, let us consider the quadratic function $g(b)=(1-\beta_2b)(1-\beta_1 b)$.  For a given value $g^*$, there are at most two different values $b_1,b_2$ such that $g(b_1)=g(b_2)=g^*$. To guarantee \eqref{eq:component-wise-term},  there are at most two different integers $i, j$ such that $c_{i,s} c_{j,t}\neq 0$, where $s\in \{1,2, \cdots, n_i\}$ and $t\in \{1,2, \cdots, n_j\}$.  If $\beta_1=0$ and $\beta_2\neq 0$, to guarantee \eqref{eq:component-wise-term},  there is only at most one integer $i \in \{1, 2, \cdots, p\}\}$ such that $c_{i,s}\neq 0$, which means that $u^*=u_p$. However, we know for such $u_p$, $(I -\alpha(u_p) A) u_p=0$ and $\Pi(u_p)$ is not defined. 
		Thus, all the eigenvectors of \eqref{eq:GEVPPi} associated with non-unit eigenvalues are any linear combination of eigenvectors of $A$ associated with exactly two distinguish eigenvalues of $A$.
	\end{proof}
	
	\begin{corollary}
		If the initial residual $r_0=u^*$ is given by \eqref{eq:all-eig-form}, then the corresponding root-convergence factor is
		\begin{equation*}
			\varrho_k(r_0)=\left(\frac{(a_{i_2}-a_{i_1})^2}{a_{i_1}^2+\epsilon^2a_{i_2}^2}\frac{1}{1+\frac{1}{\epsilon_{i_1 i_2}^2}}\right)^{1/2},
		\end{equation*}
		where $\epsilon_{i_1 i_2}^2 =\left(\frac{c_{i_2}}{c_{i_1}}\right)^2$ with $c_{i_1}$ and $c_{i_2}$ given by \eqref{eq:coeff-square-require}.
	\end{corollary}
	
	\begin{theorem}
		Suppose that  $A$ is symmetric. If $A$ is indefinite,
		\begin{equation}
			\max_{v\neq 0} \frac{\|\Pi(v)v\|}{\|v\|}=1.
		\end{equation}
		If $A$ is definite,	
		\begin{equation*}
			\max_{v\neq 0} \frac{\|\Pi(v)v\|}{\|v\|}=(\varrho^*_{GMRES(1)})^2,
		\end{equation*}
		where $\varrho^*_{GMRES(1)}$  is given by \eqref{eq:gam1-form}.
	\end{theorem}
	\begin{proof}
		Let  $(I-\alpha(v)A)v=v-\frac{\langle Av,v \rangle}{\langle Av,Av\rangle}Av=\Phi(v)$. If $\Phi(v)=0$, $\Pi(v)v$ is not defined.  Using \eqref{eq:vector-matrix-norm}, we have
		\begin{align*}
			\max_{v\neq 0}\frac{\|\Pi(v)v\|}{\|v\|} 
			& = \max_{v\neq 0}\frac{\|(I-\alpha(\Phi(v))A)\Phi(v)\|}{\|v\|}\\
			& = \max_{v\neq 0}\frac{\|(I-\alpha(\Phi(v))A)\Phi(v)\|}{ \|\Phi(v)\|} \frac{ \|\Phi(v)\|}{\|v\|}\\
			& \leq  \max_{y\neq 0} \frac{\|\mathcal{I}(y)y\|}{\|y\|} \max_{v\neq 0}\frac{\|\mathcal{I}(v)v\|}{\|v\|} \\
			&= \varrho^*_{GMRES(1)} \varrho^*_{GMRES(1)}=(\varrho^*_{GMRES(1)})^2.
		\end{align*}
		Using \eqref{eq:GEVPPi} and Lemma \ref{lem:maxim-form}, we have 
		\begin{equation*}
			\max_{v\neq 0}\frac{\|\Pi(v)v\|}{\|v\|} \geq \max_{\lambda} {\|\lambda\|}=(\varrho^*_{GMRES(1)})^2.
		\end{equation*} 
		Thus, we obtain the desired result.
	\end{proof}

	\begin{theorem}
		Consider GMRES(1) applied to $Ax=b$, where $A\in\mathbb{R}^{2\times 2}$ is symmetric. Let $r_0=c_1u_1+c_2u_2$ be the initial residual. Then, if $c_1c_2=0$,
		\begin{equation*}
			\varrho(r_0)=0.
		\end{equation*}
		If $c_1c_2\neq 0$, let $\epsilon=\frac{c_2}{c_1}$. Then
		\begin{equation}\label{eq:Gammaro-general}
			\varrho(r_0)=\sqrt{\lambda_{12}(\epsilon)} =\left(\frac{(a_2-a_1)^2}{(a_1^2+\epsilon^2a_2^2) (1+\frac{1}{\epsilon^2})}\right)^{1/2}.
		\end{equation}
	\end{theorem}

	\begin{theorem}
		For GMRES(1) applied to $Ax=b$, where $A$ is symmetric, the worst-case root-convergence factor is
		\begin{equation}\label{eq:lower-bound}
			\varrho^*_{GMRES(1)} =\max_{i\neq j, i,j\in \{1,2, \cdots,p\}} \frac{|a_j-a_i|}{|a_j|+|a_i|}
			=\left|\frac{\lambda_{\max}(A)- \lambda_{\min}(A)}{\lambda_{\max}(A)+\lambda_{\min}(A)}\right|,
		\end{equation}
		if $A$ is definite,  and $\varrho^*_{GMRES(1)}=1$ if $A$ is indefinite.
	\end{theorem}
	\begin{proof}
		Consider $r_0$ such that $\Pi(r_0)r_0=(\Gamma^*)^2r_0$, where $\Gamma^*$ is defined in \eqref{eq:def-Gamma}. Then, we have $r_{2k}=(\Gamma^*)^{2k}r_0$. It follows that $\varrho(r_0)=\lim\sup_{k\rightarrow \infty} \|r_{2k}\|^{1/(2k)}=\Gamma^*$.
	\end{proof}
	When $A$ is definite, if the initial residual is only associated with the first $s$ eigenvalues $\{a_i\}_{i=1}^{s}$, then the root-convergence factor is given by \eqref{eq:lower-bound} with $i, j \in \{1, 2, \cdots, s\}$. Although when $A$ is indefinite, the worst-case root-convergence factor is one, there may exist some good initial guesses such that the corresponding root-convergence can be less than one.

	\subsection{Some results of rAA(1)}
	
	Anderson acceleration is a widely used method to speed up the convergence of fixed-point iterations and it has garnered a lot of attention for its convergence analysis \cite{walker2011anderson,toth2015convergence,anderson1965iterative, evans2020proof,sterck2021asymptotic}.
	
	In \cite{de2024convergence}, the authors study the asymptotic root-convergence factor of restarted AA(1), denoted as rAA(1). For completeness, we give rAA(1) iterations as follows:
	\begin{align*}
		x_{k+1}& = Mx_k + b, \quad  k = 0, 2, 4, \cdots \\
		x_{k+1} & = Mx_k + \frac{\langle Ar_k, r_k \rangle}{\langle Ar_k,Ar_k\rangle}M(x_k- x_{k-1}) + b,\quad \text{where} \,\, r_k=Ax_k-b, \,\,  k = 1, 3, 5, \cdots
	\end{align*}
	It can be shown that rAA(1) residuals satisfy $r_{k+2}=M(I -\alpha(r_k) A)r_k$. To analyze its asymptotic convergence factor, the authors of \cite{de2024convergence} consider the following vector-dependent matrix
	\begin{equation*}
		\Upsilon(v) =M(I -\alpha(\Psi(v)) A) M(I -\alpha(v) A),
	\end{equation*}
	where
	\begin{equation*}
		\Psi(v) = M(I -\alpha(v) A)v, \quad v\neq 0,
	\end{equation*}
	and the corresponding nonlinear eigenvalue problem is
	\begin{equation}\label{eq:GEVPUP}
		\Upsilon(u)u=\mu u.
	\end{equation}
	The residuals of rAA(1) satisfy $r_{4(k+1)}=\Upsilon(r_{4k})r_{4k}$. There remain questions about the eigenpairs of \eqref{eq:GEVPUP}. \cite{de2024convergence} proposes some conjectures. We will derive more results of rAA(1) when $A$ is symmetric and address one of the conjectures from \cite{de2024convergence}.
	
	We first consider the following nonlinear eigenvalue problem:
	\begin{equation}\label{eq:Psi-eig}
		\Psi(u)u = M(I -\alpha(u) A)u=\mu u.
	\end{equation}
	We are interested in eigenpairs ($\mu \neq 0$) of \eqref{eq:Psi-eig}. We have the following results.
	
	\begin{theorem}\label{thm:rAA1-Psi-eigpair}
		Suppose that $A$ is symmetric. Let  $i_1, i_2 \in \{1,2,\cdots, p\}$ with $i_1\neq i_2$ and $\frac{-(a_{i_2}-1)a_{i_1}}{(a_{i_1}-1)a_{i_2}}>0$. Consider 
		\begin{equation*} 
			u^*= \sum_{s=1}^{n_{i_1}} c_{i_1,s} u_{i_1,s} +  \sum_{r=1}^{n_{i_2}} c_{i_2,r} u_{i_2,r}, \quad c_{i_1,1} c_{i_2,1}\neq 0,
		\end{equation*}
		and 
		\begin{equation*} 
			u_{i_1 i_2}=   c_{i_1} u_{i_1,s_1} + c_{i_2} u_{i_2,r_1}=c_{i_1}( u_{i_1,s_1} + \epsilon u_{i_2,r_1}),
		\end{equation*}
		where	$\epsilon=\frac{c_{i_2}}{c_{i_1}}, \epsilon^2=\frac{-(a_{i_2}-1)a_{i_1}}{(a_{i_1}-1)a_{i_2}}$, and
		\begin{equation*} 
			c_{i_1}^2=\sum_{s=1}^{n_{i_1}}c_{i_1,s}^2, \quad c_{i_2}^2=\sum_{r=1}^{n_{i_1}}c_{i_1,r}^2,
			\quad s_1 \in \{1,2,\cdots, n_{i_1}\}, \quad s_2 \in \{1,2,\cdots, n_{i_2}\}.
		\end{equation*}
		Then, $u^*$ and $u_{i_1 i_2}$ are eigenvectors of \eqref{eq:Psi-eig} associated with the same eigenvalue, i.e.,
		\begin{equation} \label{eq:eig-function-Psi}
			\Psi(u^*)u^*=\mu_{i_1 i_2} u^*, \quad \Psi(u_{i_1 i_2})u_{i_1 i_2}=\mu_{i_1 i_2} u_{i_1 i_2},
		\end{equation}
		where 
		\begin{equation}\label{eq:mui1i2-Psi}
			\mu_{i_1 i_2}(\epsilon_{i_1 i_2}) =\frac{(1-a_{i_1})(1-a_{i_2})(a_{i_1}-a_{i_2})}{(a_{i_1}-1)a_{i_1} -(a_{i_2}-1)a_{i_2}}.
		\end{equation}
		Moreover,  all the eigenpairs ($\mu \neq 0$) of \eqref{eq:Psi-eig} are given by $(\mu_{i_1,i_2}, u^*)$.
	\end{theorem}
	\begin{proof}
		From Theorem \ref{thm:multiple-two-eigpair-I}, we have $\alpha(u^*)=\alpha(u_{i_1 i_2}):=\beta$.
		Using  the second equation in \eqref{eq:eig-function-Psi} gives us 
		\begin{equation*}
			(1-a_{i_1})(1-\beta a_{i_1}) c_{i_1}u_{i_1,1}=\mu_{i_1 i_2}c_{i_1}u_{i_1,1}\quad \text{and}\,\,
			(1-a_{i_2})(1-\beta a_{i_2}) c_{i_2}u_{i_2,1}=\mu_{i_1 i_2}c_{i_2}u_{i_2,1}.	
		\end{equation*}
		Since $c_{i_1}c_{i_2}\neq 0$,
		\begin{equation}\label{eq:Psi-quad}
			(1-a_{i_1})(1-\beta a_{i_1})  =(1-a_{i_2})(1-\beta a_{i_2}).	
		\end{equation}
		Solving the above equation for $\beta$, we obtain $\beta=\frac{1}{a_{i_2}+a_{i_1}-1}$. Recall $\beta=\frac{a_{i_1}+a_{i_2}\epsilon^2}{a_{i_1}^2+a_{i_2}^2\epsilon^2}$. We further obtain $\epsilon^2=\frac{-(a_{i_2}-1)a_{i_1}}{(a_{i_1}-1)a_{i_2}}$.
		
		Thus,
		\begin{align*}
			\Psi(u^*)u^*&=\sum_{s=1}^{n_{i_1}}(1-a_{i_1}) (1-\beta a_{i_1}) c_{i_1,s}u_{i_1,s}+\sum_{r=1}^{n_{i_2}} (1-a_{i_2})(1-\beta a_{i_2}) c_{i_2,r}u_{i_2,r}\\
			&=  \sum_{s=1}^{n_{i_1}} \mu_{i_1 i_2} c_{i_1,s}u_{i_1,s}+\sum_{r=1}^{n_{i_2}} \mu_{i_1 i_2} c_{i_2,r}u_{i_2,r}\\
			&= \mu_{i_1 i_2}\left(\sum_{s=1}^{n_{i_1}}  c_{i_1,s}u_{i_1,s}+\sum_{r=1}^{n_{i_2}}  c_{i_2,r}u_{i_2,r}\right)\\
			&=\mu_{i_1 i_2} u^*.
		\end{align*}
		To satisfy \eqref{eq:Psi-quad}, there are at most two different $a_{i_1}, a_{i_2}$ existing. Thus, all eigenpairs with $\mu\neq 0$ of  \eqref{eq:Psi-eig} are given by $(\mu_{i_1,i_2}, u^*)$. 
	\end{proof}
	In the above theorem, we require that $\frac{-(a_{i_2}-1)a_{i_1}}{(a_{i_1}-1)a_{i_2}}>0$. However, if the eigenvalues of $A$ do not satisfy this condition, then it means that \eqref{eq:Psi-eig} does not have nonzero eigenvalue. This motivates us to examine the eigenpairs of \eqref{eq:GEVPUP} below.

	\begin{theorem}\label{thm:rAA1-two-eigpair}
		Suppose that $A$ is symmetric. Let $i_1, i_2 \in \{1,2,\cdots, p\}$ with $i_1\neq i_2$, and consider
		\begin{equation*} 
			u^*= \sum_{s=1}^{n_{i_1}} c_{i_1,s} u_{i_1,s} +  \sum_{r=1}^{n_{i_2}} c_{i_2,r} u_{i_2,r}, \quad c_{i_1,1} c_{i_2,1}\neq 0,
		\end{equation*}
		and 
		\begin{equation*} 
			u_{i_1 i_2}=   c_{i_1} u_{i_1,s_1} + c_{i_2} u_{i_2,r_1},
		\end{equation*}
		where	
		\begin{equation}\label{eq:sum-squares-rAA} 
			c_{i_1}^2=\sum_{s=1}^{n_{i_1}}c_{i_1,s}^2, \quad c_{i_2}^2=\sum_{r=1}^{n_{i_1}}c_{i_1,r}^2,
			\quad s_1 \in \{1,2,\cdots, n_{i_1}\}, \quad s_2 \in \{1,2,\cdots, n_{i_2}\}.
		\end{equation}
		Then, $u^*$ and $u_{i_1 i_2}$ are eigenvectors of \eqref{eq:GEVPUP} associated with the same eigenvalue, i.e.,
		\begin{equation} \label{eq:eig-function}
			\quad \Upsilon(u_{i_1 i_2})u_{i_1 i_2}=\mu_{i_1 i_2} u_{i_1 i_2},
		\end{equation}
		and
		\begin{equation}\label{eq:Upsilon-eig-combination}
			\Upsilon(u^*)u^*=\mu_{i_1 i_2} u^*,
		\end{equation}
		where 
		\begin{equation}\label{eq:mui1i2}
			\mu_{i_1 i_2}(\epsilon_{i_1 i_2}) =\frac{(a_{i_2}-a_{i_1})^2}{a_{i_1}^2+\epsilon_{i_1 i_2}^2a_{i_2}^2}\frac{((1-a_{i_1})(1-a_{i_2}))^2}{(1-a_{i_1})^2+\frac{(1-a_{i_2})^2}{\epsilon^2}}, \quad \epsilon_{i_1 i_2} =\frac{c_{i_1}}{c_{i_2}},
		\end{equation}
		and 
		\begin{equation}\label{eq:maxmui1i2}
			\max_{i_1 i_2}\mu_{i_1 i_2}(\epsilon_{i_1 i_2}) 
			=\mu_{i_1 i_2}\left(\pm \sqrt{\left|\frac{(1-a_{i_2})a_{i_1}}{(1-a_{i_1})a_{i_2}}\right|}\right)
			= \left( \frac{(1-a_{i_1})(1-a_{i_2})(a_{i_1}-a_{i_2})}{|(1-a_{i_1})a_{i_1}| +|(1-a_{i_2})a_{i_2}| } \right)^2.
		\end{equation}
	\end{theorem}
	\begin{proof}
		From Theorem \ref{thm:multiple-two-eigpair-I}, we have $\alpha(u^*)=\alpha(u_{i_1 i_2})$.	Consider
		\begin{equation*}
			u^*_1 = \Psi(u^*)=\sum_{s=1}^{n_{i_1}} (1-a_{i_1})(1-\alpha(u^*)a_{i_1}) c_{i_1,s} u_{i_1,s} + \sum_{r=1}^{n_{i_2}} (1-a_{i_2})(1-\alpha(u^*)a_{i_2}) c_{i_2,r} u_{i_2,r},
		\end{equation*}
		and 
		\begin{equation*}
			u_{i_1 i_2,1}=\Psi(u_{i_1 i_2})=(1-a_{i_1})(1-\alpha(u^*)a_{i_1}) c_{i_1} u_{i_1,1} + (1-a_{i_2}) (1-\alpha(u^*)a_{i_2}) c_{i_2}u_{i_2,1}.
		\end{equation*}
		Using condition \eqref{eq:sum-squares-rAA}, we have
		\begin{align*}
			& \sum_{s=1}^{n_{i_1}} (1-a_{i_1})^2(1-\alpha(u^*)a_{i_1})^2 c_{i_1,s}^2 + \sum_{r=1}^{n_{i_2}} (1-a_{i_2})^2(1-\alpha(u^*)a_{i_2})^2 c_{i_2,r}^2\\
			& = 
			(1-a_{i_1})^2(1-\alpha(u^*)a_{i_1})^2 c_{i_1}^2+  (1-a_{i_2})^2(1-\alpha(u^*)a_{i_2})^2 c_{i_2}^2.
		\end{align*}
		It follows that $\alpha(u^*_1)=\alpha( u_{i_1 i_2,1})$. For simplicity, let $\beta_2=\alpha(u^*_1)$ and $\beta_1=\alpha(u^*)$.
		The results \eqref{eq:eig-function},   \eqref{eq:mui1i2}, and \eqref{eq:maxmui1i2} are from \cite{de2024convergence}. Using  \eqref{eq:eig-function} gives us 
		\begin{align*}
			(1-a_{i_1})^2(1-\beta_2 a_{i_1})(1-\beta_1a_{i_1}) c_{i_1}u_{i_1,1}&=\mu_{i_1 i_2}c_{i_1}u_{i_1,1}, \\
			(1-a_{i_2})^2(1-\beta_2 a_{i_2})(1-\beta_1a_{i_2}) c_{i_2}u_{i_2,1}&=\mu_{i_1 i_2}c_{i_2}u_{i_2,1}.	
		\end{align*}
		Since $c_{i_1} c_{i_2}\neq 0$,
		\begin{equation*}
			(1-a_{i_1})^2(1-\beta_2 a_{i_1})(1-\beta_1a_{i_1})   =\mu_{i_1 i_2}, \quad
			(1-a_{i_2})^2(1-\beta_2 a_{i_2})(1-\beta_1a_{i_2})  =\mu_{i_1 i_2}.	
		\end{equation*}
		Similar to the proof of Theorem \ref{thm:rAA1-Psi-eigpair}, we can obtain \eqref{eq:Upsilon-eig-combination}.
	\end{proof}
	From Theorems \ref{thm:rAA1-Psi-eigpair} and \ref{thm:rAA1-two-eigpair}, we see that the eigenvectors of \eqref{eq:Psi-eig} are also eigenvectors of \eqref{eq:GEVPUP}.
	
	\begin{remark}
		It can be shown that when the maximum of $\mu_{i_1,i_2}(\epsilon_{i_1 i_2})$ is achieved, see \eqref{eq:maxmui1i2}, $\beta_1=\beta_2$.
	\end{remark}
	\begin{theorem}\label{thm:all-eig-Upsilon}
		Suppose that  $A$ is symmetric. Then, all the eigenvectors of \eqref{eq:GEVPUP} associated with nonzero eigenvalues are given by 
		\begin{equation}\label{eq:all-eig-form-rAA}
			u= \sum_{i\in \{1,2,\cdots, p\}} \sum_{s=1}^{n_{i}} c_{i,s} u_{i,s}, 
		\end{equation}
		where $2\leq \#\{i\}\leq 4$,  and at least one term in $\{c_{i_1,s}\}_{s=1}^{n_{i_1}}$ is nonzero and at least one term in $\{c_{i_2,r}\}_{r=0}^{n_{i_2}}$ is nonzero. 
	\end{theorem}
	The formula \eqref{eq:all-eig-form-rAA} means that the eigenvectors of \eqref{eq:GEVPUP} are any linear combination of eigenvectors of $A$ associated with at least two and at most four  distinguish eigenvalues of $A$.
	
	\begin{proof}
		As shown in Theorem \ref{thm:all-eigv-form},  $u_*=\sum_{s=1}^{n_i}c_{i,s}u_{i,s}\neq 0$, where  $i \in \{1,2,\cdots, p\}$  cannot be eigenvectors of \eqref{eq:GEVPUP} ($\alpha(0)$ is not defined).  Assume that the eigenvector of \eqref{eq:GEVPUP} is given by
		\begin{equation}\label{eq:nonzero-ci-rAA}
			u^*=\sum_{i=1}^p  \sum_{s=1}^{n_{i}} c_{i,s} u_{i,s},
		\end{equation}
		which contains at least two different integers $i,j$ such that $c_{i,s} c_{j,t}\neq 0$. 
		Next, we prove that in \eqref{eq:nonzero-ci-rAA}, there are at most four different integers $i$.  For simplicity, let $\beta_1=\alpha(u^*)$ and $\beta_2=\alpha(M(I -\alpha(u^*) A))u^*)$. We know that $\beta_1\beta_2\neq 0$. Using $\Upsilon (u^*)u^*= \lambda  u^*$ with $\lambda\neq 0$,
		we have
		\begin{equation}\label{eq:component-wise-term-rAA}
			c_{i,s}(1-a_i)^2(1-\beta_2 a_i) (1-\beta_1 a_i)u_{i,s}=\lambda c_{i,s}u_{i,s}, \forall (i,s)\in \{1,2,\cdots, p\} \times \{1,2,\cdots, n_i\}.  
		\end{equation}
		Consider the quartic function $g(b)=(1-\beta_2b)(1-\beta_1 b)(1-b)^2$.  For a given value $g^*$, there are at most four different values $b_1,b_2$ such that $g(b_1)=g(b_2)=g(b_3)=g(b_4)=g^*$. To guarantee \eqref{eq:component-wise-term},  there are at most four different integers $i, j, k, \ell$ such that $c_{i,s_1} c_{j,t_1}c_{k,s_2} c_{\ell,t_2}\neq 0$, where $s\in \{1,2, \cdots, n_i\}$ and $t\in \{1,2, \cdots, n_j\}$.
		
		%
		
		Thus, all the eigenvectors of \eqref{eq:GEVPUP} associated with nonzero eigenvalues are any linear combination of eigenvectors of $A$ associated with at least two and at most four distinguish eigenvalues of $A$.
	\end{proof}
	We comment that in Theorem \ref{thm:all-eig-Upsilon} we do not know whether \eqref{eq:all-eig-form-rAA} can contain three or four eigenvectors associated with three or four distinguish eigenvalues of $A$. If there exists, then what the corresponding eigenvalues of $\Upsilon$ are. We leave these questions for the future.
	
	Define 
	\begin{equation*}
		\Lambda^*=\max_{i,j \in \{1, 2, \cdots, p\} } \left( \frac{(1-a_{i})(1-a_{j})(a_{i}-a_{j})}{|(1-a_{i})a_{i}| +|(1-a_{j})a_{j}| } \right)^{2}.
	\end{equation*}
	In \cite{de2024convergence}, it has shown that worst-case convergence factor of rAA(1) has a lower bound $(\Lambda^*)^{1/4}$. The numerical results in that work seem to indicate that $(\Lambda^*)^{1/4}$ is the worst-case convergence factor.
	Moreover, it includes a conjecture stating that
	\begin{equation}\label{eq:conjecture-w}
		\max_{v\neq 0} \frac{\|\Upsilon(v)v\|}{\|v\|}= \Lambda^*.
	\end{equation}
	In this following, we provide three counterexamples for \eqref{eq:conjecture-w}. We consider $A$ being $\mathcal{A}_i$ as follows.
	\begin{example}[counterexamples]  
		\begin{equation*}
			\mathcal{A}_1=	\begin{pmatrix}
				1 & 0 &0\\
				0 & 2 & 0\\
				0 & 0 & 3
			\end{pmatrix},\quad
			\mathcal{A}_2=	\begin{pmatrix}
				-1/2 & 0 &0\\
				0 & -4 & 0\\
				0 & 0 & 2
			\end{pmatrix},\quad
			\mathcal{A}_3=	\begin{pmatrix}
				1/2 & 0 &0 &0\\
				0 & 3/2 & 0 &0\\
				0 & 0 & 1/3 & 0\\
				0& 0& 0 & -2
			\end{pmatrix}.
		\end{equation*}
	\end{example}
	First, it can be shown that $\Lambda^*(\mathcal{A}_1)=\frac{1}{16}=0.0625$. Consider $v_1=[15,\quad  5,  \quad 1]^T$ and $M_1=I-\mathcal{A}_1$. Then,
	\begin{align*}
		&	\alpha(v_1)=\frac{v_1^T\mathcal{A}_1v_1}{v_1^TA_1^T\mathcal{A}_1v_1}=\frac{139}{167}, \quad u=M_1(I-\alpha(v_1)\mathcal{A}_2)v_1=\left[0,  \quad \frac{555}{167}, \quad  \frac{500}{167}\right]^T,\\
		&	\alpha(u)=\frac{u^T\mathcal{A}_1u}{u^T\mathcal{A}_1^T\mathcal{A}_1u}=\frac{\frac{ 1366050  }{27889}}{\frac{
				3482100  }{27889}}=\frac{1366050}{3482100}=\frac{27321}{69642},\\
		&	w=M_1(I-\alpha(u)\mathcal{A}_1)u=\left[0, \quad \frac{-1387500}{1938369}, \quad \frac{205350}{193837}\right]\approx[0 \quad  -0.7158 \quad 1.0594]^T.
	\end{align*}
	Thus, $ \frac{\|\Upsilon(v_1)v_1\|}{\|v_1\|}\approx\frac{ 1.2786}{15.8430}\approx 0.0807>\Lambda^*(\mathcal{A}_1)=0.0625$. For more accuracy, one may verify if $\|\Upsilon(v_1)v_1\|^2=\|v_1\|^2(\Lambda^*)^2$, that is, $\|w\|^2 16^2=\|v_1\|^2$, which gives 
	$$16^2( (1387500\times193837)^2+ (205350\times1838369)^2)=(15^2+5^2+1^2)(1938369\times193837)^2.$$
	We notice that the left-hand side is even, while the right-hand side is odd; therefore, they cannot be equal.
	
	Consider $v_2=[38,\quad  1,  \quad 45]^T$ for $\mathcal{A}_2$ and  $v_3=[23, \quad 60, \quad 77, \quad 1]^T$ for $\mathcal{A}_3$. It can be shown that
	\begin{align*}
		\Lambda^*(\mathcal{A}_2)&= 225/121  \approx   1.8595, \quad \frac{\|\Upsilon(v_2)v_2\|}{\|v_2\|}\approx 2.1598>\Lambda^*(\mathcal{A}_2),\\
		\Lambda^*(\mathcal{A}_3)&=\frac{49}{81}\approx 0.6049,  \quad \frac{\|\Upsilon(v_3)v_3\|}{\|v_3\|}\approx 0.6388>\Lambda^*(\mathcal{A}_3).
	\end{align*}
	
	Although $\displaystyle\max_{v\neq 0} \frac{\|\Upsilon(v)v\|}{\|v\|}>\Lambda^*$, this does not indicate that the worst-case root-convergence of  rAA(1) is larger or equals  $\left(\max_{v\neq 0} \frac{\|\Upsilon(v)v\|}{\|v\|}\right)^{1/4}$. Here, we cannot have a similar Lemma \ref{lem:RQ-max} for rAA(1) since $\alpha(v)$ for rAA(1) does not minimize the residual.  Note that for $\mathcal{A}_1$,  $ w=\Upsilon(v_1)v_1$ is an eigenvector of \eqref{eq:GEVPUP}. From Theorem \ref{thm:rAA1-two-eigpair}, we have $\Upsilon(w)w=\mu w$. This means that for $\mathcal{A}_1^{-1}v_1$ as an initial guess for solving $\mathcal{A}_1x=0$, the corresponding root-convergence factor of rAA(1) is $\varrho=(\mu)^{1/4}<\left(\max_{v\neq 0} \frac{\|\Upsilon(v)v\|}{\|v\|}\right)^{1/4}$.

	\section{Skew-symmetric case of $M$ for GMRES(1)}\label{sec:skew}
	In this section, we discuss the situation that $M=I-A \in \mathbb{R}^{n\times n}$ is skew-symmetric, i.e., $M=-M^T$. For simplicity, we assume that $n$ is even. Then, the eigenvalues of $M$ are imaginary.  Denote the eigenvalues of $M$ as $\{m_j,\bar{m}_j\}_{j=0}^{n/2}$, where $\bar{m}_j$ is the conjugate of $m_j$. We define
	\begin{equation*}
		m_*=\min_{j=1,2,\cdots,n/2}\{|m_j|\}, \quad m^*=\max_{j=1,2,\cdots,n/2}\{|m_j|\}.
	\end{equation*}
	Let $ M=QH Q^T$ be the real Schur decomposition of $M$, where $H\in \mathbb{R}^{n\times n}$ is a block diagonal matrix with the diagonal blocks of size $2\times 2$. The eigenvalues of the $2\times 2$ diagonal blocks are the eigenvalues of $M$.  $Q$ is an orthogonal matrix with
	\begin{equation}\label{eq:schur-complement}
		Q=[Q_1,\cdots, Q_{n/2}], \quad Q_j=[q_j,\tilde{q}_j].
	\end{equation}
	From \cite{de2024convergence}, we know that $M$ is invariant on the subspaces span$\{Q_j\}$ for $j=1,2,\cdots, n/2$ and the eigenvalues of $S=\frac{1}{2}(A^{-1}+A^{-T})$ are $\{ (1+|m_j|^2)^{-1}\}_{j=1}^{n/2}$. We cite some results from \cite{de2024convergence}, which will be used for our analysis.
	\begin{lemma} \cite{de2024convergence}
		Suppose that $M$ is skew-symmetric and $q\in {\rm span}\{Q_j\}$. Then $\alpha(q)=\frac{1}{1+|m_j|^2}$, and
	 for general $0\neq v\in \mathbb{R}^n$, we have $\alpha(v) \in \left[\frac{1}{1+(m^*)^2}, \frac{1}{1+(m_*)^2}\right]\subseteq (0,1].$
	\end{lemma}
	
	\begin{theorem}\label{thm:skew-general-Phi-bound}
		Suppose that $M$ is skew-symmetric. If $0 \neq q \in {\rm span}\{Q_j\}$,  then
		\begin{equation*}
			\Phi(q) \in {\rm span}\{Q_j\}\quad \text{and} \quad \frac{\|\Phi(q)\|}{\|q\|}=\frac{|m_j|}{\sqrt{1+|m_j|^2}}.
		\end{equation*}
		Moreover,
		\begin{equation}\label{eq:q-max}
			\max_{v\neq 0} \frac{\|\Phi(v)\|}{\|v\|}\leq \frac{m^*}{\sqrt{1+(m^*)^2}}.
		\end{equation}
	\end{theorem}
	\begin{proof}
		The proof can be derived following \cite{de2024convergence}. Thus, we omit it here.
	\end{proof}
	
	\begin{corollary}\label{cor:MRI-skew-worst-case}
		Consider GMRES(1) applied to linear system $Ax=b$ with $M=I-A$ being skew-symmetric. Then, the worst-case root-convergence factor of GMRES(1) is 
		\begin{equation}\label{eq:r-factor-M-skew}
			\varrho^*_{GMRES(1),ss}=\frac{m^*}{\sqrt{1+(m^*)^2}}<1.
		\end{equation}
	\end{corollary}
	\begin{proof}
		Recall that $r_k=\Phi(r_{k-1})$. From Theorem \ref{thm:skew-general-Phi-bound}, we have
		\begin{equation*}
			\|r_k\| \leq  \frac{m^*}{\sqrt{1+(m^*)^2}}\|r_{k-1}\|\leq \cdots \leq\left( \frac{m^*}{\sqrt{1+(m^*)^2}}\right)^k \|r_0\|.
		\end{equation*}
		From the first result in Theorem \ref{thm:skew-general-Phi-bound}, when $r_0 \in {\rm span}\{Q_j\}$, we  have $r_k=\Phi(r_{k-1}) \in {\rm span}\{Q_j\}$. Thus,
		\begin{equation*}
			\varrho^*_{GMRES(1),ss}=\max_{r_0} (\lim \sup_{k\rightarrow \infty} \Gamma_k(r_0))=\frac{m^*}{\sqrt{1+(m^*)^2}},
		\end{equation*}
		which completes the proof.
	\end{proof}
	The above result tells us that GMRES(1) always converges if $M$ is skew-symmetric. In \cite{de2024convergence}, the authors consider rAA(1), and show that for the case where $M$ is skew-symmetric, the worst-case root-convergence factor of rAA(1) is $\varrho^*_{rAA(1)}= \frac{m^*}{(1+(m^*)^2)^{1/4}}$, which requires $m^*<\sqrt{\frac{1}{2}(1+\sqrt{5})}\approx 1.272$ to guarantee that rAA(1) converges for arbitrary initial guess.  Notice that no matter what $m^*$ is, we always have $\varrho^*_{GMRES(1),ss}<\varrho^*_{rAA(1)}$.
	
	From the above discussion, we can conclude that the asymptotic convergence factor of GMRES(1) for $M$ being skew-symmetric depends solely on the dominant eigenvalue (largest magnitude) with which the initial residual $r_0$ is associated. 
	\begin{corollary}\label{coroll:eig-space}
		Suppose $M$ is skew-symmetric, and the initial residual $r_0=\sum_{j=1}^{j^*}u_j$, where $u_j\in {\rm span}\{Q_j\}$, $1\leq j^*\leq n/2$, and $|m_{j^*}|$ is the largest magnitude of the eigenvalues. Then, 
		\begin{equation*}
			\varrho(r_0)=\frac{|m^*_j|}{\sqrt{1+|m^*_j|^2}}.
		\end{equation*}
	\end{corollary}

	\begin{remark}
		In \cite{saad2003iterative}, it has shown that when $A$ is a real positive definite matrix, then
		\begin{equation*}
			\|r_k\| \leq \sqrt{1-\lambda_{\min} \left(\frac{A^{-1}+A^{-T}}{2}\right) \lambda_{\min} \left(\frac{A+A^{T}}{2}\right) }\|r_{k-1}\|.
		\end{equation*}
		In fact, when $M$ is skew-symmetric, $\frac{A+A^T}{2}=\frac{I-M+ I-M^T}{2}=I$, and $\frac{A^{-1}+A^{-T}}{2}=(I+MM^T)^{-1}$.
		Then, $\lambda_{\min} \left(\frac{A^{-1}+A^{-T}}{2}\right)=(1+(m^*)^2) )^{-1}$. Thus,
		\begin{equation*}
			\sqrt{1-\lambda_{\min} \left(\frac{A^{-1}+A^{-T}}{2}\right) \lambda_{\min} \left(\frac{A+A^{T}}{2}\right)} =\sqrt{1-\frac{1}{1+(m^*)^2}},
		\end{equation*}
		which is the same result as we obtained in Corollary \ref{cor:MRI-skew-worst-case}.
	\end{remark}
	\subsection{q-linear convergence}
	Based on the previous discussion, we are able to determine the q-linear convergence of GMRES(1). The following results have been obtained.
	\begin{theorem}
		For GMRES(1) applied to $Ax=b$, where $A$ is symmetric, the residuals converge  q-linearly with q-factor $\sigma$:
		\begin{equation}
			\|r_k\|\leq \sigma \|r_{k-1}\|,
		\end{equation}
		where $\sigma=\rho^*_{GMRES(1)}$ is given by \eqref{eq:lower-bound} if $A$ is symmetric and $\sigma=\rho^*_{GMRES(1),ss}$ is given by \eqref{eq:r-factor-M-skew} if $A$ is skew-symmetric.
	\end{theorem}
	\begin{proof}
		These results can be derived directly from Lemma \ref{lem:RQ-max} and \eqref{eq:q-max}.
	\end{proof}
	We point out that the q-linear convergence with q-factor $\sigma$ actually is the worst-case rate. For different initial guesses/errors,  $\sigma$ might be different, but it must be less than  $\rho^*_{GMRES(1)}$ ($\rho^*_{GMRES(1),ss}$). However, we do not further discuss this. In \cite{vecharynski2010cycle}, it has been shown the sequence $\left\{\frac{\|r_k\|}{\|r_{k-1}\|}\right\}$ is increasing. We know that $\frac{\|r_k\|}{\|r_{k-1}\|}$ has an upper bound $\sigma$, so the sequence must be convergent.
	
	\section{Numerical experiments}\label{sec:num}
	In this section, we present some numerical experiments to validate our  theoretical results of worst-case root-convergence factor of GMRES(1).  For symmetric $A$, we consider three choices, which correspond to three cases of  $M$: $M$ is not invertible, $\rho(M)<1$, and $\rho(M)>1$.
	\begin{example}[$A$ is symmetric]
		We consider solving $Ax=b=0$, where $A$ is  $A_i$.
		\begin{equation*}
			A_1=	\begin{pmatrix}
				1 & 0 & 0 \\
				0  &2 & 0\\
				0 & 0 & 3
			\end{pmatrix},\,
			A_2=	\begin{pmatrix}
				1/2 & 0 & 0 & 0 &0 \\
				0  &1/4 & 0 & 0 & 0\\
				0 & 0 & 1/8 & 0 & 0\\
				0 & 0 &0  &1/16 &0\\
				0 & 0 &0  &0 &1/32
			\end{pmatrix},
			\,
			A_3=	\begin{pmatrix}
				-1 & 0 & 0  &0\\
				0  &2 & 0 & 0\\
				0 & 0 & 3 & 0\\
				0 & 0 & 0 & 4\\
			\end{pmatrix}.
		\end{equation*}
	\end{example}
	Using \eqref{eq:gam1-form}, we obtain the worst-case root-convergence factor for $A_1, A_2,A_3$: $\varrho^*_1=\frac{1}{2}=0.5,  \varrho^*_2=\frac{15}{17}\approx 0.8824, \varrho^*_3=1$ for $A_1, A_2$ and $A_3$, respectively.
	For our tests, we run 1000 random initial guesses, whose elements are uniform distribution in the interval $(-1,1)$, and we require $\|r_k\|\leq 1e-30$ to make $k$ large enough to capture the behavior of $\varrho_k(r_0)=\|r_k\|^{1/k}$.  Figure \ref{fig:A1A2} reports the root-convergence factor as a function of the iteration index $k$ for different initial guesses for $A_1$ and $A_2$. The dashed line is our theoretical worst-case root-convergence factor, denoted as $\varrho^*$. We see that the numerical root-convergence factor is bounded by the worst-case root-convergence factor, and the root-convergence factor depends on initial guesses. 
		\begin{figure}[H]
		\centering
		\includegraphics[width=5cm]{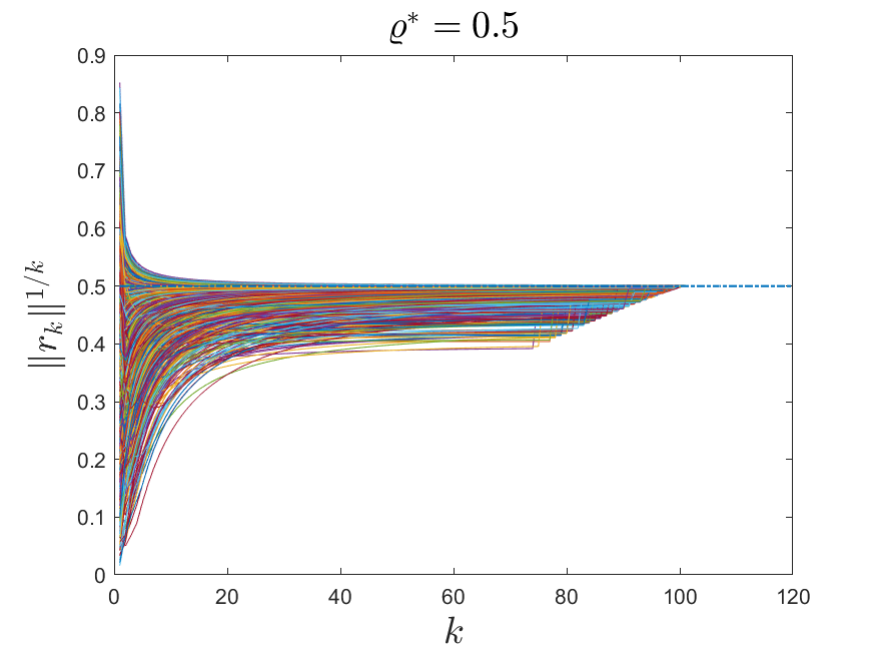}
		\includegraphics[width=5cm]{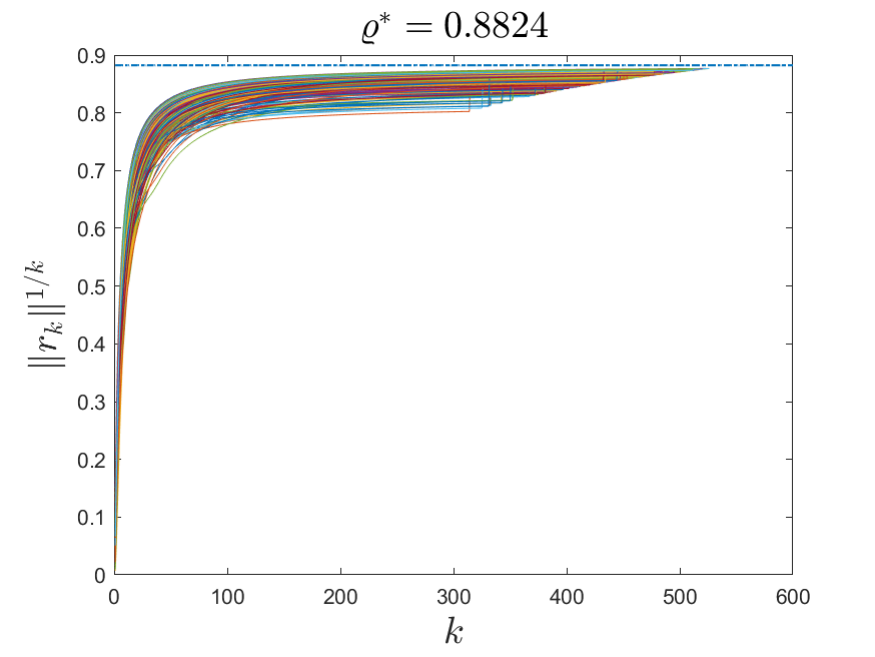}
		\caption{Root-convergence factor $\varrho_k(r_0)=\|r_k\|^{1/k}$ (solid lines) as a function of the iteration index $k$ with 1000 random initial guesses for $A_1$ (left) and $A_2$ (right). The dashed line is the theoretical worst-case root convergence factor, denoted as $\varrho^*$.}\label{fig:A1A2}
	\end{figure}
	
	For $A_3$, the indefinite case, the root-convergence factor is displayed on the left side of Figure \ref{fig:A3two}. To verify our result in Lemma \ref{lem:eigpair-two-condition}, which states the convergence factor depends on the initial guess,  we also consider $A_3$ and choose initial guesses where the first component is zero. Then, the worst-case root-convergence factor depends only on the eigenvalues 2, 3, 4. Using Theorem \ref{thm:GMRES1-worst-case}, we have $\varrho^*=\frac{4-2}{4+2}=\frac{1}{3}\approx0.3333$. The result on the right side of Figure \ref{fig:A3two} validates this.
		\begin{figure}[H]
		\centering
		\includegraphics[width=5cm]{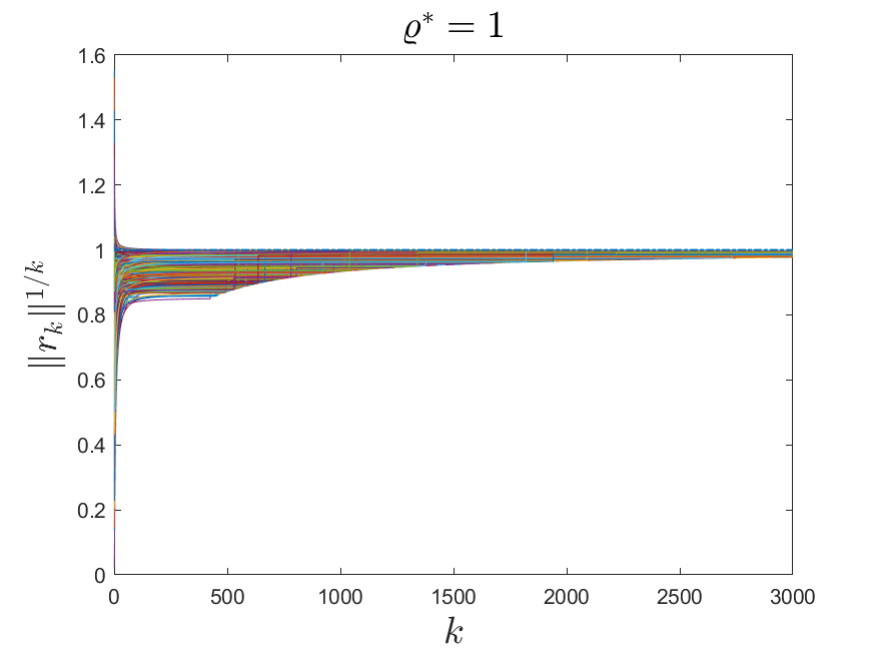}
		\includegraphics[width=5cm]{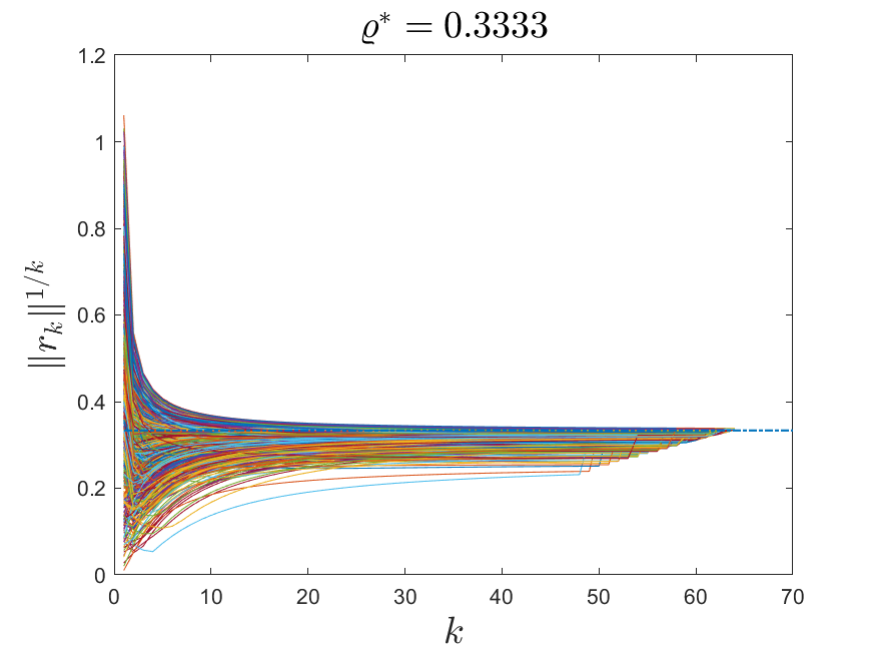}
		\caption{Root-convergence factor $\varrho_k(r_0)=\|r_k\|^{1/k}$ (solid lines) as a function of the iteration index $k$ with 1000 random initial guesses for  $A_3$ (left)  and  for $A_3$ (right) with the first elements of the initial guesses being zero. The dashed line is the theoretical worst-case root-convergence factor, denoted as $\varrho^*$.}\label{fig:A3two}
	\end{figure}
	
	To investigate how the initial guess affects the root-convergence factor, see Theorem \ref{thm:two-eigenvectors}, we consider two different initial guesses for $A_2$ (denoted its eigenvalues as $a_1=1/2, a_2=1/4, a_3=1/16$).  One initial guess is  $x_0=[1, 2\sqrt{2}, 0,0,0]^T$ corresponding to $r_0=[1/2, \sqrt{2}/2, 0,0,0]^T$. Then,  $\epsilon=\frac{c_2}{c_1}=\sqrt{2}=\sqrt{\frac{a_1}{a_2}}$. We have $\varrho(r_0)=\frac{a_1-a_2}{a_1+a_2}=\frac{1}{3}$. The another is $x_0=[1, 0, 8,0,0]^T$ corresponding to $r_0=[1/2, 0, 1,0,0]^T$. Then, $\epsilon=\frac{c_3}{c_1}=2=\sqrt{\frac{a_1}{a_3}}$ and $\varrho^*=\frac{a_1-a_3}{a_1+a_3}=0.6000$. The root-convergence factors are reported in Figure \ref{fig:A2initial}. The dashed line is the theoretical worst-case root-convergence factor. We observe that the numerical root-convergence factors align with our theoretical results.

	\begin{figure}
		\centering\includegraphics[width=5cm]{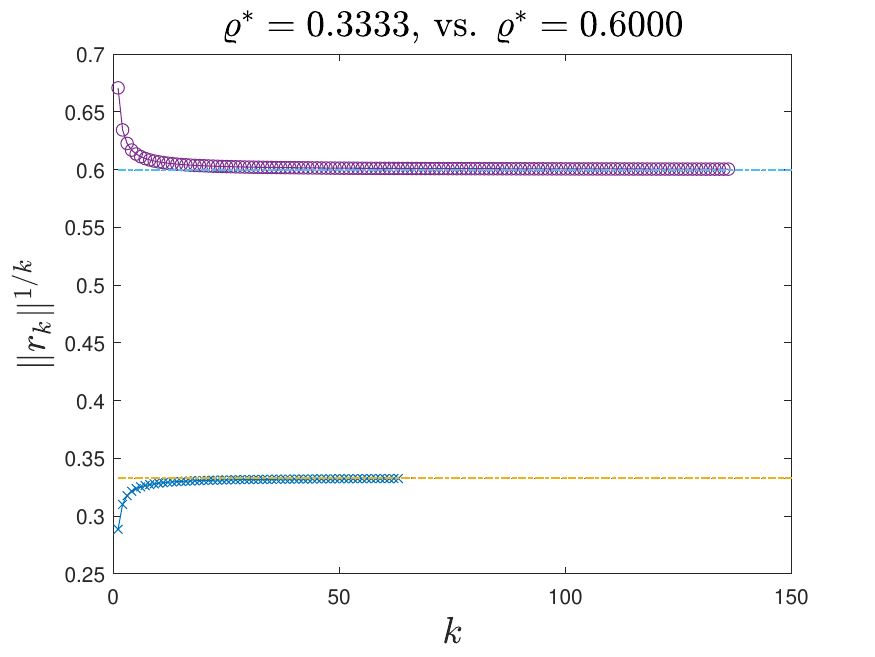}
		\caption{Root-convergence factor $\varrho_k(r_0)=\|r_k\|^{1/k}$ (solid lines) as a function of the iteration index $k$ with two different initial guesses for $A_2$. The dashed line is the theoretical worst-case root-convergence factor. Cross markers  corresponds to $x_0=[1, 2\sqrt{2}, 0,0,0]^T$, where $\varrho^*=0.3333$ . Circle markers are for $x_0=[1, 0, 8,0,0]^T$, where $\varrho^*=0.6000$.}\label{fig:A2initial}
	\end{figure}
	In the following example, we consider the case that $M$ is symmetric.
	\begin{example}[$M$ is skew-symmetric]
		Consider $A_4=I-M$, where 
		\begin{equation*}
			M= \frac{1}{8}	\begin{pmatrix}
				0 & 1 & 0 & -5 & 0&  0 & 0 & 2\\ 
				-1 & 0 & 0 & 0 &  5&  0 &  -2& 0\\ 
				0 & 0 & 0 & 0 & -2 & -1 & 5 & 0 \\
				5 & 0 & 0 & 0 & -1 & -2 & 0 & 0\\ 
				0 &-5 & 2 & 1 & 0 & 0 & 0 & 0\\ 
				0 &0 & 1 & 2 & 0 & 0 & 0 & -5\\ 
				0 &2 & -5 & 0 & 0 & 0 & 0 & 1\\ 
				-2& 0 & 0 & 0 & 0& 5 & -1 & 0
			\end{pmatrix},
		\end{equation*}
		which is taken from \cite{ward1976eigensystem} with some modification. The eigenvalues of $M$ are $\pm\iota\frac{1}{4}, \iota \pm\frac{2}{4},  \pm\iota \frac{3}{4}, \pm\iota 1$. Thus, $m^*=1$ and $\varrho^*=\frac{1}{\sqrt{1+1^2}}\approx 0.7071$.
	\end{example}
	On the left side of Figure \ref{fig:A4}, we run 1000 random initial guesses, whose elements are uniform distribution in the interval $(-1,1)$, and report the root-convergence factor $\varrho_k(r_0)=\|r_k\|^{1/k}$ as a function of the iteration index $k$. The dashed line is our theoretical worst-case root-convergence factor, denoted as $\varrho^*$. We see that when $k \rightarrow \infty$, the root-convergence factor converges to the worst-case root-convergence factor $\varrho^*=0.7071$. To validate Corollary \ref{coroll:eig-space}, we consider $x_0=Q(:,3)+Q(:,5)+Q(:,7)$, see \eqref{eq:schur-complement}, which is a vector in the space spanned by the columns of $Q_2, Q_3, Q_4$ corresponding to eigenvalues $\pm\iota\frac{1}{4}, \iota \pm\frac{2}{4},  \pm\iota \frac{3}{4}$ of $M$. Thus, the expected root-convergence factor is $\frac{3/4}{\sqrt{1+(3/4)^2}}=\frac{3}{5}=0.6$. Our numerical results on the right side of Figure \ref{fig:A4} confirm this.
	
	\begin{figure}
		\centering
		\includegraphics[width=5cm]{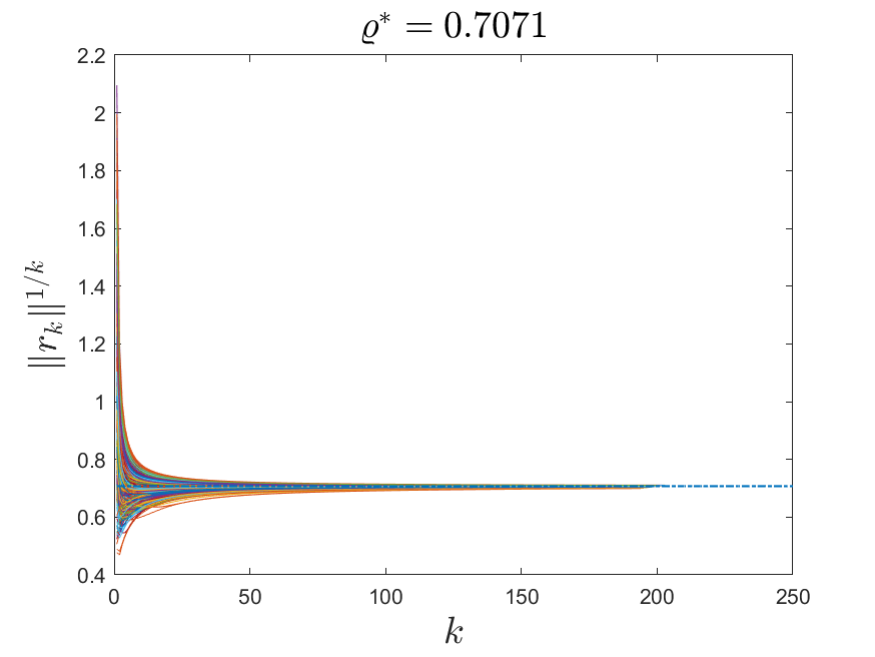}
		\includegraphics[width=5cm]{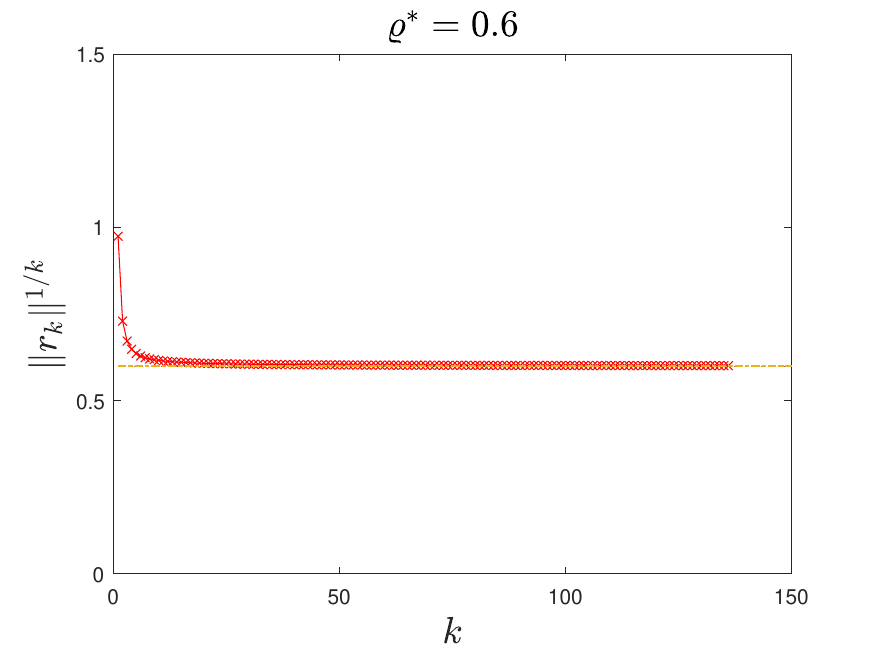}
		\caption{Root-convergence factor $\varrho_k(r_0)=\|r_k\|^{1/k}$ (solid lines) as a function of the iteration index $k$ with 1000 random initial $x_0$ for $A_4$ (left). The dashed line is the theoretical worst-case root convergence factor, denoted as $\varrho^*$. Cross markers are  $\varrho_k(r_0)=\|r_k\|^{1/k}$ as a function of the iteration index $k$ with $x_0=Q(:,3)+Q(:,5)+Q(:,7)$ for $A_4$ (right). }\label{fig:A4}
	\end{figure}

	\section{Conclusion}\label{sec:con}
	We study the root-convergence factor of GMRES(1) in solving linear systems. We derive theoretical results on the worst-case root-convergence factor based on eigenvector-dependent nonlinear eigenvalue problems. When the coefficient matrix $A$ of the underlying linear system is symmetric, the root-convergence factor highly depends on the initial guess. For some special initial guesses, we are able to derive their root-convergence factors. When the matrix $M=I-A$ is skew-symmetric, our theoretical results reveal that the worst-case root-convergence factor depends solely on the spectral radius of $M$, and GMRES(1) converges unconditionally. The q-linear convergence factor is derived too, which is the same as the worst-case root-convergence factor.  Additionally, we present new results for restarted Anderson acceleration with window size one.

\section*{Acknowledgments}
The author thanks Microsoft Copilot for spelling and grammar check.

\bibliographystyle{siam}
\bibliography{gmresBib}

\begin{thebibliography}{10}

\bibitem{agullo2014block}
{\sc E.~Agullo, L.~Giraud, and Y.-F. Jing}, {\em Block {GMRES} method with
  inexact breakdowns and deflated restarting}, SIAM Journal on Matrix Analysis
  and Applications, 35 (2014), pp.~1625--1651.

\bibitem{anderson1965iterative}
{\sc D.~G. Anderson}, {\em Iterative procedures for nonlinear integral
  equations}, Journal of the ACM (JACM), 12 (1965), pp.~547--560.

\bibitem{arioli2009analysis}
{\sc M.~Arioli}, {\em An analysis of {GMRES} worst case convergence},
  Unpublished report,  (2009), p.~240.

\bibitem{baker2005technique}
{\sc A.~H. Baker, E.~R. Jessup, and T.~Manteuffel}, {\em A technique for
  accelerating the convergence of restarted {GMRES}}, SIAM Journal on Matrix
  Analysis and Applications, 26 (2005), pp.~962--984.

\bibitem{both2019anderson}
{\sc J.~W. Both, K.~Kumar, J.~M. Nordbotten, and F.~A. Radu}, {\em Anderson
  accelerated fixed-stress splitting schemes for consolidation of unsaturated
  porous media}, Computers $\&$ Mathematics with Applications, 77 (2019),
  pp.~1479--1502.

\bibitem{de2024convergence}
{\sc H.~De~Sterck, O.~A. Krzysik, and A.~Smith}, {\em Asymptotic convergence of
  restarted anderson acceleration for certain normal linear systems}, arXiv
  preprint arXiv:2312.04776v3,  (2024).

\bibitem{embree2003tortoise}
{\sc M.~Embree}, {\em The tortoise and the hare restart {GMRES}}, SIAM review,
  45 (2003), pp.~259--266.

\bibitem{embree2022descriptive}
\leavevmode\vrule height 2pt depth -1.6pt width 23pt, {\em How descriptive are
  {GMRES} convergence bounds?}, arXiv preprint arXiv:2209.01231,  (2022).

\bibitem{essai1998weighted}
{\sc A.~Essai}, {\em Weighted {FOM} and {GMRES} for solving nonsymmetric linear
  systems}, Numerical Algorithms, 18 (1998), pp.~277--292.

\bibitem{evans2020proof}
{\sc C.~Evans, S.~Pollock, L.~G. Rebholz, and M.~Xiao}, {\em A proof that
  {A}nderson acceleration improves the convergence rate in linearly converging
  fixed-point methods (but not in those converging quadratically)}, SIAM
  Journal on Numerical Analysis, 58 (2020), pp.~788--810.

\bibitem{faber2013properties}
{\sc V.~Faber, J.~Liesen, and P.~Tichy}, {\em Properties of worst-case
  {GMRES}}, SIAM Journal on Matrix Analysis and Applications, 34 (2013),
  pp.~1500--1519.

\bibitem{frommer1998restarted}
{\sc A.~Frommer and U.~Gl{\"a}ssner}, {\em Restarted {GMRES} for shifted linear
  systems}, SIAM Journal on Scientific Computing, 19 (1998), pp.~15--26.

\bibitem{giraud2010flexible}
{\sc L.~Giraud, S.~Gratton, X.~Pinel, and X.~Vasseur}, {\em Flexible {GMRES}
  with deflated restarting}, SIAM Journal on Scientific Computing, 32 (2010),
  pp.~1858--1878.

\bibitem{greenbaum1996any}
{\sc A.~Greenbaum, V.~Pt{\'a}k, and Z.~e.~k. Strako{\v{s}}}, {\em Any
  nonincreasing convergence curve is possible for {GMRES}}, SIAM Journal on
  Matrix Analysis and Applications, 17 (1996), pp.~465--469.

\bibitem{greenbaum1994gmres}
{\sc A.~Greenbaum and L.~N. Trefethen}, {\em {GMRES/CR} and {A}rnoldi/{L}anczos
  as matrix approximation problems}, SIAM Journal on Scientific Computing, 15
  (1994), pp.~359--368.

\bibitem{guttel2014some}
{\sc S.~G{\"u}ttel and J.~Pestana}, {\em Some observations on weighted
  {GMRES}}, Numerical Algorithms, 67 (2014), pp.~733--752.

\bibitem{joubert1994convergence}
{\sc W.~Joubert}, {\em On the convergence behavior of the restarted {GMRES}
  algorithm for solving nonsymmetric linear systems}, Numerical Linear Algebra
  with Applications, 1 (1994), pp.~427--447.

\bibitem{liesen2000computable}
{\sc J.~Liesen}, {\em Computable convergence bounds for {GMRES}}, SIAM Journal
  on Matrix Analysis and Applications, 21 (2000), pp.~882--903.

\bibitem{liesen2004convergence}
{\sc J.~Liesen and Z.~Strakos}, {\em Convergence of {GMRES} for tridiagonal
  {T}oeplitz matrices}, SIAM Journal on Matrix Analysis and Applications, 26
  (2004), pp.~233--251.

\bibitem{liesen2004worst}
{\sc J.~Liesen and P.~Tich{\`y}}, {\em The worst-case {GMRES} for normal
  matrices}, BIT Numerical mathematics, 44 (2004), pp.~79--98.

\bibitem{lin2012simpler}
{\sc Y.~Lin, L.~Bao, and Q.~Wu}, {\em Simpler {GMRES} with deflated
  restarting}, Mathematics and Computers in Simulation, 82 (2012),
  pp.~2238--2252.

\bibitem{morgan1995restarted}
{\sc R.~B. Morgan}, {\em A restarted {GMRES} method augmented with
  eigenvectors}, SIAM Journal on Matrix Analysis and Applications, 16 (1995),
  pp.~1154--1171.

\bibitem{morgan2002gmres}
\leavevmode\vrule height 2pt depth -1.6pt width 23pt, {\em {GMRES} with
  deflated restarting}, SIAM Journal on Scientific Computing, 24 (2002),
  pp.~20--37.

\bibitem{nachtigal1992hybrid}
{\sc N.~M. Nachtigal, L.~Reichel, and L.~N. Trefethen}, {\em A hybrid {GMRES}
  algorithm for nonsymmetric linear systems}, SIAM Journal on Matrix Analysis
  and Applications, 13 (1992), pp.~796--825.

\bibitem{ortega2000iterative}
{\sc J.~M. Ortega and W.~C. Rheinboldt}, {\em Iterative solution of nonlinear
  equations in several variables}, SIAM, 2000.

\bibitem{robbe2002convergence}
{\sc M.~Robb{\'e} and M.~Sadkane}, {\em A convergence analysis of {GMRES} and
  {FOM} methods for {S}ylvester equations}, Numerical Algorithms, 30 (2002),
  pp.~71--89.

\bibitem{saad2000further}
{\sc Y.~Saad}, {\em Further analysis of minimum residual iterations}, Numerical
  Linear Algebra with Applications, 7 (2000), pp.~67--93.

\bibitem{saad2003iterative}
\leavevmode\vrule height 2pt depth -1.6pt width 23pt, {\em Iterative methods
  for sparse linear systems}, SIAM, 2003.

\bibitem{saad1986gmres}
{\sc Y.~Saad and M.~H. Schultz}, {\em {GMRES}: A generalized minimal residual
  algorithm for solving nonsymmetric linear systems}, SIAM Journal on
  Scientific and Statistical Computing, 7 (1986), pp.~856--869.

\bibitem{simoncini1996hybrid}
{\sc V.~Simoncini and E.~Gallopoulos}, {\em A hybrid block {GMRES} method for
  nonsymmetric systems with multiple right-hand sides}, Journal of
  Computational and Applied Mathematics, 66 (1996), pp.~457--469.

\bibitem{sterck2021asymptotic}
{\sc H.~D. Sterck and Y.~He}, {\em On the asymptotic linear convergence speed
  of {A}nderson acceleration, {N}esterov acceleration, and nonlinear {GMRES}},
  SIAM Journal on Scientific Computing, 43 (2021), pp.~S21--S46.

\bibitem{titley2014gmres}
{\sc D.~Titley-Peloquin, J.~Pestana, and A.~J. Wathen}, {\em {GMRES}
  convergence bounds that depend on the right-hand-side vector}, IMA Journal of
  Numerical Analysis, 34 (2014), pp.~462--479.

\bibitem{toh1997gmres}
{\sc K.-C. Toh}, {\em {GMRES} vs. ideal {GMRES}}, SIAM Journal on Matrix
  Analysis and Applications, 18 (1997), pp.~30--36.

\bibitem{toth2015convergence}
{\sc A.~Toth and C.~T. Kelley}, {\em Convergence analysis for {A}nderson
  acceleration}, SIAM Journal on Numerical Analysis, 53 (2015), pp.~805--819.

\bibitem{van1993superlinear}
{\sc H.~A. Van~der Vorst and C.~Vuik}, {\em The superlinear convergence
  behaviour of {GMRES}}, Journal of Computational and Applied Mathematics, 48
  (1993), pp.~327--341.

\bibitem{van1994gmresr}
{\sc H.~A. Van~der Vorst and C.~Vuik}, {\em {GMRESR}: a family of nested
  {GMRES} methods}, Numerical Linear Algebra with Applications, 1 (1994),
  pp.~369--386.

\bibitem{vecharynski2010cycle}
{\sc E.~Vecharynski and J.~Langou}, {\em The cycle-convergence of restarted
  {GMRES} for normal matrices is sublinear}, SIAM Journal on Scientific
  Computing, 32 (2010), pp.~186--196.

\bibitem{walker2011anderson}
{\sc H.~F. Walker and P.~Ni}, {\em Anderson acceleration for fixed-point
  iterations}, SIAM Journal on Numerical Analysis, 49 (2011), pp.~1715--1735.

\bibitem{ward1976eigensystem}
{\sc R.~Ward and L.~Gray}, {\em Eigensystem computation for skew-symmetric
  matrices and a class of symmetric matrices.[subroutines trizd, imzd, and
  tbakzd in fortran for ibm 360/91 computer]}, tech. rep., Oak Ridge National
  Lab.(ORNL), Oak Ridge, TN (United States), 1976.

\bibitem{zhong2008complementary}
{\sc B.~Zhong and R.~B. Morgan}, {\em Complementary cycles of restarted
  {GMRES}}, Numerical Linear Algebra with Applications, 15 (2008),
  pp.~559--571.

\bibitem{zitko2000generalization}
{\sc J.~Z{\'\i}tko}, {\em Generalization of convergence conditions for a
  restarted {GMRES}}, Numerical Linear Algebra with Applications, 7 (2000),
  pp.~117--131.

\end{thebibliography}

\end{document}